\documentclass[reqno]{amsart}[14]
\usepackage{amssymb, amsmath}
\usepackage[foot]{amsaddr}
\allowdisplaybreaks %long equations
\usepackage{mathrsfs}
\usepackage{comment}
\usepackage{color}
\usepackage{tikz-cd}

\numberwithin{equation}{section} %enumerate equations with section number

\def\sideremark#1{\ifvmode\leavevmode\fi\vadjust{\vbox to0pt{\vss% the remark
 \hbox to 0pt{\hskip\hsize\hskip1em%                          will appear only
 \vbox{\hsize3cm\tiny\raggedright\pretolerance10000%          on the side
 \noindent #1\hfill}\hss}\vbox to8pt{\vfil}\vss}}}%

\theoremstyle{plain} \newtheorem{thm}{Theorem}[section]
\theoremstyle{plain} 
\theoremstyle{plain} \newtheorem{lem}[thm]{Lemma}
\theoremstyle{plain} \newtheorem{prop}[thm]{Proposition}
\theoremstyle{plain} \newtheorem{cor}[thm]{Corollary}
\theoremstyle{definition} \newtheorem{defn}{Definition}
\theoremstyle{plain} \newtheorem{rema}[thm]{Remark}
\newtheorem{exas}[thm]{Examples}

\newcommand{\be}{\begin{equation}}
\newcommand{\ee}{\end{equation}}
\newcommand{\bea}{\begin{eqnarray}}
\newcommand{\eea}{\end{eqnarray}}
\newcommand{\beas}{\begin{eqnarray*}}
\newcommand{\eeas}{\end{eqnarray*}}

\newcommand{\R}{\mathbb{R}}

\newcommand{\N}{\mathbb{N}}
\newcommand{\Z}{\mathbb{Z}}
\newcommand{\Zs}{\Z_2} 		%super
\newcommand{\Zn}{\Zs^n} %multigraded (changed to eliminate parentheses)

\newcommand{\cO}{{\mathcal{O}}} %sheaf
\newcommand{\cJ}{\mathcal{J}}
\newcommand{\cT}{\mathcal{T}}

\newcommand{\cB}{{\mathcal{B}}}
\newcommand{\cF}{{\mathcal{F}}}
\newcommand{\cG}{{\mathcal{G}}}

\newcommand{\Ci}{{\mathcal{C}}^{\infty}} %smooth

\newcommand{\Om}{{\Omega}}

\newcommand{\op}[1]{\!\!\mathop{\rm ~#1}\nolimits}

\newcommand{\I}{\mathbb{I}}			%id matrix

\DeclareMathOperator{\Der}{\textrm{Der}} 		%Derivations
\DeclareMathOperator{\supp}{\textrm{supp }}	%support
\newcommand{\ev}{\textrm{ev}}								%evalutaion map
\newcommand{\iHom}{\underline{Hom}} % internal hom (sheaf)
\DeclareMathOperator{\Hom}{Hom}								%categorical hom
\newcommand{\GL}{\mathsf{GL}^{\mathbf{0}}}		% invertible (0-deg) matrices
\newcommand{\gl}{\mathsf{gl}}									%graded matrices
\newcommand{\ModA}{\Zn\mbox{-}\mathtt{Mod}_{A}} %A-module--needs revision?
\newcommand{\ots}{\otimes}
\newcommand{\vect}[1]{{\bf #1}}    %vector--needs revision?
\newcommand{\mathsc}[1]{{\mathscr #1}}   %mathsc--needs revision?
\newcommand{\mathcs}[1]{{\mathscr #1}}

\begin{document}
%%%%%%%%%%%%%%%%%%%%%%%%%%%%%%%%%%%%%%%%%%%%%%%%%%%%%%%%%%%%%%%
%%%%%%%%%%%%%%%%%%%%%%%%%%%%%%%%%%%%%%%%%%%%%%%%%%%%%%%%%%%%%%%

%------------------------vertical space text-equation-text
\belowdisplayskip=7pt plus 3pt minus 4pt
\belowdisplayshortskip=7pt plus 3pt minus 4pt
%---------------------------------------------------------

%authors

\author{Tiffany Covolo$^1$, Stephen Kwok $^2$,}
\author{Norbert Poncin $^2$}
\address{$^1$Faculty of Mathematics, National Research University Higher School of Economics, 7 Vavilova Str., 117312 Moscow, Russia}
\address{$^2$Mathematics Research Unit, University of Luxembourg, 6, rue Richard Coudenhove-Kalergi, L-1359, Luxembourg}
\email{tcovolo@hse.ru, stephen.kwok@uni.lu, norbert.poncin@uni.lu}

\title{Differential calculus on $\mathbb{Z}^n_2$-supermanifolds}
\maketitle

\begin{abstract}
The concept of $\Zn$-supermanifold has been recently proposed as a natural generalization of classical ($\Zs$-graded) supergeometry, allowing for more complicated commutativity constraints. Here we continue the study of $\Zn$-supergeometry by developing the foundations of differential calculus on $\Zn$-supermanifolds.
\end{abstract}

%-%-%-%-%-%-%-%-%-%-%-%-%-%-%
\section*{Introduction}
%-%-%-%-%-%-%-%-%-%-%-%-%-%-%

Recently, there has been interest in generalizing supergeometry by allowing notions of commutativity with respect to more general grading groups. Motivation for this generalization comes from various areas, including physics. For example, more general gradings appear in string theory \cite{AFT} and in parastatistics, $\Z_2^2$-graded Lie algebras make an appearance \cite{YJ}. $\Zn$-supermanifolds also play a role in geometric mechanics \cite{GGU}. It was also suggested \cite{COP} that one could gain further insight into Clifford algebras by exploiting the fact that they carry a natural structure of commutative $\Zn$-graded algebra.

In the 1980s, Molotkov \cite{Mo} defined $\Zn$-supermanifolds in an abstract, infinite-dimensional setting, taking them to be locally modeled on locally convex topological vector spaces. More recently, a Berezin-Leites style definition of finite-dimensional $\Zn$-supermanifolds, in terms of ringed spaces endowed with structure sheaves of commutative $\Zn$-graded local rings, was given in \cite{CGPa}.

This paper continues the project begun in \cite{CGPa} and develops the basic theory of differential calculus on $\Zn$-supermanifolds. We have striven to make it readable for those who may not be familiar with classical supergeometry, so have included a fair amount of detail. Specialists in supergeometry will note that the theorems proven here often closely parallel their analogues in $\Zs$-supergeometry, with one crucial difference: the presence of formal power series in the coordinate rings of $\Zn$-supermanifolds necessitates the constant invocation of the Hausdorff-completeness of the $\cJ$-adic topology. Indeed, even linear algebra over a commutative $\Zn$-graded ring $R$ is not well-behaved unless $R$ is $J$-adically Hausdorff complete, and many natural operations on the sheaves (e.g. exterior derivative, wedge product) are continuous with respect to the $\cJ$-adic topology.

We briefly review the contents of this paper. The first section is dedicated to a quick review of various preliminaries: graded-commutative algebra, the category of $\Zn$-supermanifolds, the $\mathcal{J}$-adic topology, and the functor of points. Next, we discuss the tangent sheaf of a $\Zn$-supermanifold and the tangent space at a point, and then we define the modified Jacobian of a morphism and prove the chain rule. In the next section, we prove that the category of $\Zn$-supermanifolds admits finite products. After that, we discuss the local forms of morphisms of $\Zn$-supermanifolds: the crucial inverse function theorem is proven, as well as canonical local forms for immersions and submersions, and finally the constant rank theorem that shows that any morphism of constant rank may be factored into the composition of an immersion and a submersion. In the final section, we move to the cotangent sheaf, cotangent space, and differential forms. Here we show existence and uniqueness of the exterior derivative, then prove Poincar\'{e}'s lemma for $\Zn$-supermanifolds, showing that de Rham cohomology of a $\Zn$-supermanifold $M$ agrees with the ordinary cohomology of the underlying reduced space $|M|$. We treat various miscellanea, including linear algebra for $J$-adically Hausdorff complete rings, in an appendix.

\section*{Acknowledgments}

S.K. thanks the Luxembourgish National Research Foundation for its support via AFR grant no. 7718798.

%-%-%-%-%-%-%-%-%-%-%-%-%-%-%
\section{Preliminaries}\label{sec:prelim}
%-%-%-%-%-%-%-%-%-%-%-%-%-%-%

In this section, we will fix the notation used throughout the article and recall some basic definitions and results. For further details we refer the reader to the first two articles of this series on $\Zn$-graded geometry \cite{CGPa}, \cite{CGPb}, as well as previous papers of the authors on $\Zn$-graded algebra (\cite{COP},\cite{Cov},\cite{CM}), and references therein.

\subsection{Multigraded-commutative algebra}
In the sequel, $\mathbb{K}$ always denotes a field of characteristic $0$. In our notation, $\Zn=\Zs\times\cdots\times\Zs$ ($n$-times).

A \emph{$\Zn$-graded algebra} $A$ (over $\mathbb{K}$), is an $\mathbb{K}$-algebra of the form $A=\oplus_{\gamma \in\Zn} A^{\gamma}$ (decomposition into $\mathbb{K}$-vector spaces), in which the multiplication respects the $\Zn$-degree, i.e., $A^{\alpha}\cdot A^{\beta} \subset A^{\alpha+\beta}$. In this paper, we will always assume algebras to be associative and unital ($1\in A^{\mathbf{0}}$).
If in addition, for any pair of homogeneous elements $a\in A^{\alpha}$ and $b\in A^{\beta}$,
\be\label{signrule}
ab=(-1)^{\langle \alpha,\beta \rangle}\,ba \;,
\ee
where $\langle \;,\; \rangle$ denotes the usual scalar product, then the algebra $A$ is said to be \emph{$\Zn$-commutative}.

\begin{exas}
\begin{enumerate}
	\item[a)] Supercommutative algebras are the simplest examples ($n=1$).
	\item[b)] As shown in \cite{MO1, MO2}, the quaternion algebra $\mathbb{H}$ (and more generally any Clifford algebra $\op{C}\!\ell_k$) can be seen as a $\Zn$-commutative algebra for $n=3$ (respectively, $n=k+1$). For this, it suffices to associate appropriate degrees to the generators, e.g.,
	$$ \deg(\mathsf{i})= (0,1,1)\;, \qquad \deg(\mathsf{j})= (1,0,1) \qquad (\mbox{ and thus } \deg(\mathsf{k})= (1,1,0)\;).$$
\end{enumerate}
\end{exas}

One can easily observe that a $\Zn$-commutative algebra $A$ has in fact an \emph{underlying parity}, given by
$$ \Zn\ni \gamma=(\gamma_1, \ldots,\gamma_n)\; \mapsto\; \bar{\gamma}:=\sum_{k=1}^n \gamma_k \in\Zs \;.$$
In other words, degrees $\gamma$ in $\Zn$ are divided into \emph{even} ($\bar{\gamma}=\bar{0}$) and \emph{odd} ($\bar{\gamma}=\bar{1}$), which induces the analogous subdivision of the
homogeneous elements of $A$ (labeled in the same way as \emph{even} or \emph{odd}). To differentiate this underlying $\Zs$-grading from the original $\Zn$-grading we will use the intuitive notation. 
%%%%%%%%%%%%%%%%%%%%
\begin{comment}
$$ A=A_{\bar{0}} \oplus A_{\bar{1}}=\left(\oplus_{\alpha \in \Zn_\bar{0}} A^\alpha \right)\oplus\left(\oplus_{\beta \in\Zn_\bar{1}} A^{\beta}\right)\;.$$
\end{comment}
%%%%%%%%%%%%%%%%%%%%

Note that, following the generalised sign rule \eqref{signrule}, every odd-degree element of $A$ is nilpotent, as it is familiar in supergeometry. However, the higher $\Zn$-case ($n>2$), is essentially different from the super case, as in general there can be anticommuting elements which are not nilpotent. We will come back to this essential point in the next sections.

\medskip
%Hom and internal Hom
Analogously to the definition of $\Zn$-commutative algebras, other notions of linear algebras are straightforwardly inferred (see in particular \cite{CM} for a detailed exposition of the subject).
In this way, graded modules over a graded-commutative algebra $A$ and degree-preserving (right) $A$-linear  maps between them form an
 abelian category $\ModA$, 
which naturally admits a \emph{symmetric monoidal structure}  $\ots$ (see \cite[Section 2.2]{CM}),
with braiding given by
$$
\begin{array}{rccc}
c_{VW}^{gr}: &V \otimes W &\to& W \otimes V\\
&v \otimes w &\mapsto& (-1)^{\langle \deg(v),\deg(w)\rangle} w \otimes v
\end{array}
$$
for homogeneous elements $v$ and $w$.
This structure is also \emph{closed}, as for every graded $A$-module $W$, the functor $-\ots W: \ModA \to \ModA $ has a right-adjoint $\iHom_A(W,-): \ModA \to \ModA$, i.e. for any graded $A$-modules $V,U$ there is a natural isomorphism
$$ 
\Hom_A(V\ots W, U)\simeq \Hom_A(V,\iHom_A(W,U))\;.
$$
%\edz{of course, since it is symmetric monoidal, it is biclosed}
As one can readily verify, the \emph{internal hom} $\iHom_A(V,W)$ is the graded $A$-module which consists of all (right) $A$-linear maps
$  \ell: V \to W \;.$
Some may shift the $\Zn$-degree of the elements by a fixed $\gamma\in \Zn$, i.e., 
$$
\ell(V^{\alpha})\subset W^{\alpha+\gamma}\;
$$ 
for all $\alpha\in\Zn$. These latter constitute the $\gamma$-part $\iHom^{\gamma}_A(V,W)$ of $\iHom_A(V,W)$.
Hence, contrary to the case of modules over a classical commutative algebra, the internal hom $\iHom_A$ differs from the categorical hom $\Hom_A$, since this latter contains only $0$-degree $A$-linear maps. In other words, $\Hom_A(V,W)=\iHom^{0}_A(V,W)$.

\subsection{$\Zn$-supermanifolds}

The basic objects of our study are smooth $\Zn$-supermanifolds. Aside from the usual commuting coordinates (denoted in the following with the letter $x$), they present also different ``types'' of ``formal'' coordinates $\xi$ (corresponding to the different non-zero degrees in $\Zn$) which may commute or anticommute, following the generalized sign rule \eqref{signrule}. It is important to note that, contrary to superspaces (case $n=1$), in general here not every formal variable is nilpotent!

To keep track of these differences in a local coordinate system of a $\Zn$-supermanifold, we have to introduce some more notation.

Using the {underlying parity} we fix a \emph{standard order} of the elements of $\Zn$: first the even degrees ordered lexicographically, then the odd ones also ordered lexicographically. For example,
$$ (\Zs)^2=\{(0,0), (1,1), (0,1),(1,0)\}\;. $$
So that when we will refer to $\gamma_j\in\Zn\setminus\{0\}$, the \emph{$j$-th non-zero degree of $\Zn$}, we will always mean with respect to this standard order. We may thus write $\xi_{\gamma_j}$ to specify that the considered formal coordinates are of degree $\gamma_j\in\Zn\setminus\{0\}$.
Then, a tuple $\vect{q}=(q_1,\ldots,q_N)\in\R^N$ (where $N:=2^{n}-1$) provides all the information on the $\Zn$-graded variables $\xi$: there is a total of $|\vect{q}|:=\sum_{k=1}^N q_i$ graded variables $\xi^a$, among which $q_i$ of degree $\gamma_{i}\in\Zn\setminus\{0\}$, denoted by  $\xi_{\gamma_i}^{a_i}$ ($1\leq a_i\leq q_i$, $1\leq i\leq N$). For simplicity, the variables $\xi$ are also considered to be  ordered by degree.
%In this way, $\xi^1, \ldots,\xi^{q_1}$ are of degree $\gamma_1=(0, \ldots,0,1,1)$, $\xi^{q_1+1},\ldots, \xi^{q_1+q_2}$ are of degree $\gamma_2=(0, \ldots, 1,0,1)$, etc.
Hence, throughout this article, a system of coordinates of a $\Zn$-superspace will be denoted in different ways, depending on the level of distiction needed: either $u$ (no distinction between coordinates), $(x,\xi)$ (considering only the zero/non-zero degree subdivision) or $(x, \xi_{\gamma_j})$ (considering the full $\Zn$-degree subdivision).

We are now ready to recall the definition of a $\Zn$-supermanifold.
\begin{defn}

A (smooth) \emph{$\Zn$-supermanifold} of dimension $p|\vect{q}$ is a \emph{locally $\Zn$-ringed space} $(M,\cO_M)$ which is locally isomorphic to a $\Zn$-superdomain $(\R^p,\Ci_{\R^p}[[\xi]])$. Local sections of %the structure sheaf of 
this latter are \emph{formal power series} in $\Zn$-graded variables $\xi$ and smooth coefficients 
$$
 \Ci(U)[[\xi]]:=\left\{ \sum_{\alpha\in\N^{N}}^{\infty} f_{\alpha}\xi^{\alpha}\; |\; f_{\alpha}\in\Ci(U)\right\} \;.
$$

Morphisms between \emph{$\Zn$-supermanifolds} are simply morphisms of $\Zn$-ringed spaces, i.e., pairs
$(\phi,\phi^*):(M,\cO_M)\to (N,\cO_N)$ of a continuous map $\phi:M\to N$ and sheaf morphism $\phi^*_V:\cO_N(V)\to\cO_M(\phi^{-1}(V))$, $V\subset N$ open.
\end{defn}

\begin{rema}
In the following, we will use $M$ both to denote the $\Zn$-supermanifold $(M,\cO_M)$ and its base space, preferring the classical notation $|M|$ for this latter when confusion might arise, and analogously for morphisms ($\phi$ and $|\phi|$). The category of $\Zn$-supermanifolds will be denoted by $\Zn\text{-}\mathtt{Man}$.
%Also, we may sometimes drop the subscript when denoting the structure sheaf $\cO$ of $M$.
\end{rema}

\subsection{Support of a function.} The following definition is standard:

\begin{defn}
Let $U$ be an open submanifold of a $\Zn$-supermanifold $M$, and let $f \in \cO(U)$. Let $U_f$ be the set of all points $u \in U$ such that $f = 0$ on some neighborhood of $u$. The {\it support} of $f$ is $supp(f) := U \backslash U_f$. 
\end{defn}

\noindent It is not difficult to see that $U_f$ is open in $U$, hence $supp(f)$ is a closed subset of $U$.

\subsection{$\cJ$-Adic topology and Hausdorff completeness}
Let $I$ be a homogeneous ideal of a $\Zn$-graded ring $R$, $M$ an $R$-module. The collection of sets $\{x + I^kM \}_{k=0}^\infty$, where $x$ runs over all elements of $M$, is readily seen to be a basis for a topology on $M$, called the {\it $I$-adic topology}. It follows immediately from the definition that the $I$-adic topology is translation-invariant with respect to the additive group structure of $M$.

The following lemma is standard but we give its proof for completeness.

\begin{lem}\label{Iadiccont}
Let $I$ be a homogeneous ideal of $R$, and let $f: M \to N$ be an $R$-module morphism. Then $f$ is $I$-adically continuous.
\end{lem}

\begin{proof}
Let $y$ be any element of $N$. Since the sets $y + I^kN$ constitute a basis for the $I$-adic topology of $N$, it suffices to prove $f^{-1}(y + I^kN)$ is $I$-adically open in $M$ for all $k$. As $f$ is an $R$-morphism, $f(I^kM) \subseteq I^kN$. Hence $f^{-1}(y + I^kN)$ is the union of the open sets $x + I^kM$ where $x$ runs over all elements of $f^{-1}(y + I^kN)$, so that $f^{-1}(y + I^kN)$ is open.
\end{proof}

One checks that the $I$-adic topology on $R$ makes $R$ into a topological ring.

\begin{defn}
Let $I$ be a homogeneous ideal of $R$. $R$ is {\it Hausdorff complete} in the $I$-adic topology if the natural ring morphism $ R \;\xrightarrow{p}\; \varprojlim\nolimits_{k\in\N} R/I^k\;$ is an isomorphism. 
\end{defn}

We next deal with notions of convergence in the $I$-adic topology.

\begin{defn}
Let $I$ be an ideal of $R$, and let $a_i$ be a sequence of elements of $R$. $a_i$ is an {\it $I$-adic Cauchy sequence} if for each non-negative integer $n$ there exists $l$ such that $a_j - a_k \in I^n$ for all $j, k \geq l$. $a_i$ {\it converges $I$-adically} if there exists $a \in R$ such that for each non-negative integer $n$, there exists $l$ such that $a_i - a \in I^n$ for all $i \geq l$.
\end{defn}

Evidently these definitions may be extended to the $I$-adic topology on an $R$-module $M$ in a natural way. The following proposition is standard.

\begin{prop}
Let $R$ be a commutative $\Zn$-superalgebra, and $I$ a homogeneous ideal. Suppose $R$ is $I$-adically Hausdorff complete. Then a sequence $\{a_i\}$ is $I$-adically Cauchy if and only if it converges $I$-adically to a unique limit in $R$.
\end{prop}

\begin{proof} The fact that a convergent sequence is Cauchy is straightforward, so we prove the converse. Let $p$ be the canonical map from $R \to \varprojlim\nolimits_{k\in\N} R/I^k$. If $a_i$ is a Cauchy sequence, it yields a well-defined element $[a]_k$ in $R/I^k$ for each $k$: since $a_i \equiv a_j \text{ mod } I^k$ for $i, j$ sufficiently large, we define $[a]_k$ to be the equivalence class of $a_i$ for $i$ large enough. 

Noting that $\varprojlim\nolimits_{k\in\N} R/I^k = \{(b_1, b_2, \dotsc ) \in \prod_{k \geq 0} R/I^k : b_k = f_{kl}(b_l), k \leq l\}$, where $f_{kl}: R/I^l \to R/I^k$ is the obvious map, the $[a]_k$ define a unique element $a'$ of $\varprojlim\nolimits_{k\in\N} R/I^k$. Since $p$ is an isomorphism, there is a unique element $a := p^{-1}(a')$ of $R$. It is then directly verified that $a$ is the limit of the sequence $\{a_i\}$. To prove uniqueness of the limit, suppose $a_i$ converges to $\widetilde{a}$. Then $a - \widetilde{a} = (a-a_i) + (a_i - \widetilde{a})$ lies in $I^k$ for all $k$ since $a-a_i$ and $a_i - \widetilde{a}$ are both Cauchy, but $\cap_{k =0}^\infty I^k = 0$ by the Hausdorff property, whence $a = \widetilde{a}$.
\end{proof}

A similar proposition is true for $I$-adically Hausdorff complete $R$-modules; the proof is left to the reader.\\

Canonically associated to any $\Zn$-graded algebra $R$ is the homogeneous ideal $J$ of $R$ generated by all homogeneous elements of $R$ having nonzero $\Gamma$-degree. If $f: R \to S$ is a morphism of $\Zn$-graded algebras, then $f(J_R) \subseteq J_S$. The $J$-adic topology plays a fundamental role in $\Zn$-supergeometry.

This notion can be sheafified: for a $\Zn$-supermanifold $M$, we have an ideal sheaf $\cJ$, defined by $\cJ(U) = \langle f \in \cO(U) \,:\, f \text{ is of nonzero $\Zn$-degree} \rangle$, which defines a $\cJ$-adic topology on $\cO$ in an obvious way. If $\cF$ is a sheaf of $\cO$-modules, there is an analogous $\cJ$-adic topology on $\cF$, defined in a natural way. Throughout this paper, all statements about sheaves of topological $\cO$-modules (e.g., saying a sheaf is Hausdorff complete) will refer to this induced $\cJ$-adic topology.

As shown in \cite{CGPa, CGPb}, many basic results valid for smooth $\Zs$-supermanifolds also hold for this multigraded generalization. For instance, the underlying space of a $\Zn$-supermanifold $M$ admits a structure of smooth manifold $\Ci_M$, and there is a canonical projection $\varepsilon:\cO_M\to\Ci_M$, corresponding to the \emph{reduced space} of the $\Zn$-supermanifold. It can be shown that $\cJ =\ker\varepsilon$.

The obstacle, in the higher $\Zn$-case, represented by the loss of the nilpotency of $\cJ$ (a fundamental fact in supergeometry), is compensated for by the Hausdorff completeness of the $\cJ$-adic topology:

\begin{prop}[Proposition 6.9 in \cite{CGPa}]
Let $M$ be a $\mathbb{Z}^n_2$-supermanifold. Then $\cO_M$ is \emph{$\cJ$-adically Hausdorff complete} as a sheaf of $\mathbb{Z}^n_2$-commutative rings, i.e., the morphism
$$ \cO_M \;\xrightarrow{p}\; \varprojlim\nolimits_{k\in\N} \cO_M/\cJ^k\;$$
naturally induced by the filtration of $\cO_M$ by the powers of $\cJ$ is an isomorphism.\\
\end{prop}

\subsection{The functor of points}

Similar to what happens in classical $\Zs$-supergeometry, $\Zn$-superfunctions on a $\Zn$-supermanifold $M$ cannot be recovered from their values at topological points of $|M|$ because of the presence of elements of nonzero $\Zn$-degree in $\cO_M$, whose restrictions to topological points are $0$. To remedy this defect, we broaden the notion of ``points", as was suggested by Grothendieck in the context of algebraic geometry.

Let $S$, $M$ be $\Zn$-supermanifolds. An {\it $S$-point} of $M$ is just a morphism $S \to M$ of $\Zn$-supermanifolds. An $S$-point of $M$ is naturally identified with a section of the projection $M \times S \to S$, as one may readily see (we will show later that finite products of $\Zn$-supermanifolds exist, so that speaking of $M \times S$ and the projection $M \times S \to S$ has a meaning). One may regard an $S$-point of $M$ as a ``family of points of $M$, parameterized by $S$."

Let $Mor(S, M)$ denote the set of all $\Zn$-supermanifold morphisms from $S$ to $M$. A morphism $F: S \to S'$ induces a map of sets $F^*: Mor(S', M) \to Mor(S, M)$ by precomposition with $F$, and for a composition of morphisms $S'' \xrightarrow{G} S \xrightarrow{F} S'$, one has $(F \circ G)^* = G^* \circ F^*$. Thus $Mor(-, M)$ is a functor $(\Zn\text{-}\mathtt{Man})^{op} \to (Sets)$. This functor is called the {\it functor of points} of $M$, and we denote it by $\underline{M}$.

We have the following 

\begin{prop}
Let $M$ and $N$ be $\Zn$-supermanifolds. The correspondence $F \mapsto F^*$ induces a natural bijection of the set of all morphisms $M \to N$ of with the set of all natural transformations $\underline{M} \to \underline{N}$.
\end{prop} 

The proof is a formal exercise using Yoneda's lemma and exactly mimics the one in the case of classical $\Zs$-supergeometry, so will not be repeated here. This proposition allows us to reduce the study of $\Zn$-supermanifolds to the study of set-valued functors from the category of $\Zn$-supermanifolds.\\

\noindent \underline{Example}: the Chart Theorem for $\Zn$-supermanifolds (Thm. 7.10 of \cite{CGPa}) tells us that $Mor(S, \R^{p|{\bf q}})$ is the set of all ordered homogeneous $p|{\bf q}$-tuples of functions in $\Gamma(\cO_S)$. For a composition of morphisms $S'' \xrightarrow{G} S \xrightarrow{F} S'$, we have $(F \circ G)^* = G^* \circ F^*$ (here ${}^*$ denotes the pullback of global sections of the structure sheaf). Thus

\begin{align*}
&\underline{\R^{p|{\bf q}}}(S) =\\
&\{ \text{ordered homogeneous $p|{\bf q}$-tuples $(f_1, \dotsc, f_p, \nu^{\mu_1}_1, \dotsc, \nu^{\mu_1}_{j_1}, \dotsc, \nu^{\mu_q}_1, \dotsc, \nu^{\mu_q}_{j_q})$, where $f_i, \nu^\mu_j \in \Gamma(\cO_S)$} \}.
\end{align*}

%-%-%-%-%-%-%-%-%-%-%-%-%-%-%
\section{The tangent sheaf}\label{sec:tg}
%-%-%-%-%-%-%-%-%-%-%-%-%-%-%

Let $U$ be an open subset of a $\Zn$-supermanifold $M$. We consider the set $Der_\R(\cO(U))$ of $\mathbb{Z}^n_2$-graded $\R$-linear derivations on $\cO(U)$, i.e. $\R$-linear maps $D: \cO(U) \to \cO(U)$ satisfying the graded Leibniz rule:

\[
D(ab) = Da \cdot b + (-1)^{\langle \deg(D), \deg(a) \rangle} a \cdot Db\;.
\]\

Then $Der_\R(\cO(U)$) is a graded $\cO(U)$-module, as may readily be verified. We now show that derivations are local.

\medskip
For this we will require the following localization principle. We follow the discussion in \cite{Lei}.

\begin{lem}\label{localization}
Let $X$ be a closed subset of $M$, $U$ an open subset such that $X \subset U$, and $f \in \cO(U)$. Then there exists an open subset $V$ such that $X \subset V \subseteq U$, and a global section $h \in \cO(M)$ such that $f|_V = h|_V$ and $\supp(h) \subseteq \supp(f)$. If $X$ is compact, then $h$ may be taken to have compact support in $U$.
\end{lem}

\begin{proof}
By \cite[Section 7.4]{CGPa}, there exists a partition of unity $\{\varphi_\alpha\}$ subordinate to some refinement of the open cover $\{U, M \backslash X\}$ of $M$. Let $I_X$ denote the set of all $\alpha$ such that $\supp(\varphi_\alpha)$ intersects $X$; it is clear that $\supp(\varphi_\alpha) \subset U$ for all $\alpha \in I_X$. We then set
\[
h := \sum_{\alpha \in I_X} \varphi_\alpha \cdot f\;.
\]

\medskip

All terms of the sum are elements of $\cO(M)$ compactly supported in $U$, hence also in $M$. Thus $h$ so defined is also in $\cO(M)$ by local finiteness of the partition, and $\supp(h) \subseteq \supp(f)$. If $X$ is compact, $I_X$ is finite and hence $h$ has compact support in $U$.
\end{proof}

\begin{lem}
Let $M$ be a $\mathbb{Z}^n_2$-supermanifold, $D: \cO(M) \to \cO(M)$ a global derivation, and $U$ an open submanifold of $M$. Then there exists a unique derivation $D|_U: \cO(U) \to \cO(U)$ such that $(Dg)|_U = D|_U(g|_U)$.
\end{lem}

\begin{proof}
Let $V$ be an open submanifold of $M$; we show that if $h|_V \equiv 0$, then $Dh|_V \equiv 0$. Indeed, suppose $v \in V$, then by Lemma \ref{localization} there exists a function $\varphi \in \cO(M)^{0}$
 such that $supp(\varphi) \subset V$ and $\varphi \equiv 1$ on some neighborhood $W$ of $v$. Then $\varphi h \equiv 0$, whence $D(\varphi h) = D \varphi \cdot h + \varphi \cdot Dh \equiv 0$. As $h|_W \equiv 0$ and $\varphi|_W \equiv 1$, we have that $Dh|_W \equiv 0$. But this is true near any point $v \in V$, whence $Dh|_V \equiv 0$.

Now suppose $f \in \cO(U)$ and $u \in U$. By Lemma \ref{localization} there exists a function $h \in \cO(M)$ agreeing with $f$ in some neighborhood $V$ of $u$ in $U$. By the above discussion, the $Dh|_V$ are independent of the choice of $h$ and depend only on $f$. Thus, the local functions $Dh|_V$ piece together to define a unique function on $U$, which we denote by $D|_Uf$. This procedure defines an operator $D|_U$, which is readily seen to be a graded derivation from the fact that $D$ is. The uniqueness of $D|_U$ is a straightforward consequence of our proof.
\end{proof}

Hence, given an inclusion of open subsupermanifolds $V \subseteq U$, we may define a restriction homomorphism $\rho_{UV}: \Der_\R(\cO(U)) \to \Der_\R(\cO(V))$ by assigning to a derivation $X$ on $U$ the unique derivation $X|_U$ on $V$ given by the lemma. It is readily checked that the $\rho_{UV}$ so defined satisfy the axioms for the restriction homomorphisms of a sheaf of $\cO$-modules. 

Hence we may make the following definition:

\begin{defn}
The {\it tangent sheaf} $\cT M$ of a $\mathbb{Z}^n_2$-supermanifold $M$ is the sheaf of topological $\cO$-modules
$$
\cT M(U) := \Der_{\R}(\cO(U))\;,
$$
with the restriction homomorphisms $\rho_{UV}$ defined as above.
\end{defn}

\medskip
The following $\cJ$-adic continuity property of derivations will be crucial in much of what follows:
\begin{prop}\label{dercont}
Let $X$ be a homogeneous $\mathbb{Z}^n_2$-graded derivation of $\cO(U)$. Then $X: \cO(U) \to \cO(U)$ is $\cJ(U)$-adically continuous.
\end{prop}
In particular, this means that if a sequence of sections $(f_k)_{k\in\N}$ tends $\cJ$-adically to $f$, then when $k\to\infty$, $Xf_k$ tends $\cJ$-adically to $Xf$.
Clearly, a similar statement holds at the level of stalks: $X_m: \cO_{m} \to \cO_{m}$ is $\cJ_m$-adically continuous for any point $m \in M$.

\begin{proof}
We will show by induction that $X (\cJ^k(U)) \subseteq \cJ^{k-1}(U)$ for any $k$. The case $k = 1$ is vacuously true. Suppose $X(\cJ^k(U)) \subseteq \cJ^{k-1}(U)$ for some $k$. Any element of $\cJ^{k+1}(U)$ is a finite sum of elements of the form $ab$, with $a \in \cJ(U)$ and $b \in \cJ^k(U)$, $a$ homogeneous. Then $X(ab) = Xa \cdot b + (-1)^{\langle \deg(X), \deg(a) \rangle} a \cdot Xb\,$; it follows from the inductive hypothesis that the right hand side is in $\cJ^k(U)$. 

Now let $g \in \cO(U)$. Noting that $X$ is a homomorphism for the additive group structure of $\cO(U)$, the fact that $X (\cJ^{k+1}(U)) \subseteq \cJ^k(U)$ implies $X^{-1}(g + \cJ^k(U))$ is the union of the $\cJ(U)$-adically open sets $f + \cJ^{k+1}(U)$, where $f$ runs over all elements of $X^{-1}(g + \cJ^k(U))$, so that $X^{-1}(g + \cJ^k(U))$ is open. Since the sets $g + \cJ^k(U)$ are a basis for the $\cJ(U)$-adic topology, we are done.
\end{proof}

Let $U$ be an open subset of $M$. A section of $\cT M$ on $U$ is called a {\it vector field} on $U$. The real $\Zn$-graded vector space of vector fields on $U$ forms a \emph{$\mathbb{Z}^n_2$-Lie (color) algebra} %notation in [CM]
in the sense of \cite[Section 2]{COP}, with the graded Lie bracket defined by:
$$
[X, Y](f) := X(Yf) - (-1)^{\langle \deg(X),\, \deg(Y) \rangle} Y(Xf),
$$
for homogeneous $X, Y \in \cT M(U)$ and any $f \in \cO(U)$, and extended to all vector fields by linearity.
Indeed, the reader may readily check that the above defines an $\R$-bilinear operation $[- \, , -]:\cT M(U) \otimes_\R \cT M(U) \to \cT M(U)$ which is  {\it graded antisymmetric}:\

$$
[X, Y] + (-1)^{\langle \deg(X),\, \deg(Y) \rangle} [Y, X] = 0
$$
and satisfies the {\it graded Jacobi identity}:\

\[
\begin{split}
[X, [Y, Z]] + &(-1)^{\langle \deg(X),\, \deg(Y) \rangle + \langle \deg(X),\, \deg(Z) \rangle} [Y, [Z, X]]
\\ &+ (-1)^{\langle \deg(X),\, \deg(Z) \rangle + \langle \deg(Y),\, \deg(Z) \rangle} [Z, [X, Y]] = 0\;.
\end{split}
\]

\medskip
As in the ungraded case, the following proposition is basic:
\begin{prop}\label{tangentbasis}
Let $M$ be a $\mathbb{Z}^n_2$-supermanifold of dimension $p|{\bf q}$. Then $\cT M$ is a locally free sheaf of topological $\cO_M$-modules, of rank $p|{\bf q}$. More precisely, let $u=(u^i)$ be a coordinate system on an open set $U$. Then the $(\partial_{u^i})$ form an $\cO(U)$-basis of $\cT M(U)$.
\end{prop}

\begin{rema}
Consequently, the stalk $(\cT M)_m$ at any point $m\in M$ is a free $\cO_m$-module of rank $p|{\bf q}$,
\,with induced basis $\left( \left[\partial_{u^i} \right]_m \right) \,$.
\end{rema}

\begin{proof}
That the $\partial_{u^i}$ are $\cO(U)$-linearly independent is readily checked. To show that they span $\cT M(U)$, let $D$ be a derivation on $U$. Let $a^b:= Du^b$, and set $D' = D - \sum_b a^b \partial_{u^b}$. Since $D$ is a graded derivation, $D'P = 0$ for any polynomial $P$ in the $u^b$. By $\cJ$-continuity (cf. Proposition \ref{dercont}), it follows that $D'P = 0$ for any polynomial section in the sense of \cite{CGPa}, i.e. a section of $\cO(U)$ of the form $s = \sum_{|\mu| \geq 0} P_\mu(x) \xi^\mu$, where the $\mu$ are multi-indices and $P_\mu(x)$ are polynomials in the nonzero-degree coordinates $x^i$. Let $f \in \cO(U)$ and $m$ be any point in $U$. By polynomial approximation \cite[Thm. 6.10]{CGPa}, for any $k$ there exists a polynomial section $Q$ such that $[f]_m - [Q]_m \in \mathfrak{m}^k_m$, where $\mathfrak{m}_m$ denotes the unique homogeneous maximal ideal of $\cO_m$. Applying $D'$ to $f - Q$, we see that $[D'f]_m$ lies in $\mathfrak{m}^k_m$ for every $k$, hence $[D'f]_m = 0$. Since $m \in U$ was arbitrary, $D'f = 0$ for any function $f \in \cO(U)$, i.e. $D = \sum_b a^b \partial_{u^b}$.
\end{proof}

\begin{cor}
The sheaf of modules $\cT M$ is Hausdorff complete. Likewise, at every point $m\in M$, the stalk $(\cT M)_m$ is Hausdorff complete.
\end{cor}

\begin{proof}
By Proposition \ref{tangentbasis}, for any coordinate set $U$, $\cT M(U)$ is a free $\cO(U)$-module of finite rank. Since the structure sheaf $\cO$ is $\cJ$-adically Hausdorff complete, this suffices to conclude (cf. Proposition  \ref{freecomplete} in the Appendix).
\end{proof}

\medskip

\begin{comment}
%%%%definition of the pushforward (or differential) of a morphism of grd-mfds
A morphism of $\Zn$-supermanifolds $\Phi=(|\phi|,\phi^*): M \to N$ induces one on the structure sheaves of the corresponding tangent bundles  $d\Phi: \cT M \to \Phi^* \cT N$. Locally, on coordinate domains $(U, u=(x,\xi))$ of $M$ near a point $m$ and $(V,v=(y,\eta))$ of $N$ near the point $|\phi|(m)$ (small enough), this is given intuitively by 
$$
(d\phi \,\partial_{u^i})(f):=  \partial_{u^i}(\phi^*(f))\;,
$$
for any $f\in\cO_N(V)$. 
%This map is a $O_M$-module morphism, locally represented as we will see later (Chain rule) by the matrix $\partial_{u}(\phi^*v)$.
\end{comment}

%------------------------------------------------
\subsection{Tangent space and tangent map}
%------------------------------------------------

\begin{defn}
Let $M$ be a $\Zn$-supermanifold, and let $m \in M$. The {\it tangent space} to $M$ at $m$, denoted $T_mM$, is the $\mathbb{Z}^n_2$-super $\R$-vector space of graded $\R$-linear derivations $\cO_m \to \R$.
\end{defn}

\begin{prop}
Let $X$ be a vector field defined in a neighborhood of $m$. Then $X$ induces a tangent vector $X|_m$ to $M$ at $m$. If $X$ is homogeneous, the degree of $X|_m$ is the same as that of $X$.
\end{prop}

\begin{proof}
It is easily seen that $X: \cO(U) \to \cO(U)$ induces a graded derivation $X_m: \cO_m \to \cO_m$ at the stalk level such that if $X$ is homogeneous, $X_m$ has the same degree as $X$. Let $\varepsilon: \cO_m \to \Ci_m$ be the algebra morphism (of degree $0$) induced by pullback to the reduced manifold of $M$, and $\ev_m: \Ci _m \to \R$ be the evaluation morphism (of degree $0$) at $m$. We set
\be\label{tgtvecfromfield}
(X|_m)[f]_m := (\ev_m \circ \varepsilon \circ X_m)[f]_m\;.
\ee
It is readily verified that $X|_m: \cO_m \to \R$ so defined is a graded derivation having the same degree as $X$.
\end{proof}

As with the tangent sheaf, if $\dim(M) = p|{\bf q}$ then $\dim_{\R}(T_mM) = p|{\bf q}$. Indeed, given a coordinate system $(x^i, \xi^a)$ centered at $m$, the tangent vectors $\left(\partial_{x^i}|_m, \partial_{\xi^a}|_m\right)$ at $m$, induced by the coordinate vector fields following \eqref{tgtvecfromfield}, are a basis for $T_mM$. As in the case of the tangent sheaf, this is proven by polynomial approximation.

\medskip
We may compare the geometric fiber of the tangent sheaf at a point $m$ with the tangent space at $m$ defined above.
\begin{prop}
Let $m \in M$ be a point, $\mathfrak{m}$ the maximal ideal of $\cO_m$. Then,
% the tangent space $T_m M$ is isomorphic to the $\mathbb{Z}^n_2$-graded $\R$-vector space $(\cT M)_m/\left(\mathfrak{m} \cdot (\cT M)_m\right)$.
$$ T_m M \simeq (\cT M)_m / \left(\mathfrak{m}\cdot(\cT M)_m\right)$$
as $\mathbb{Z}^n_2$-super $\R$-vector space.
\end{prop}
\begin{proof}
This follows by unraveling the proof of the previous proposition.
\end{proof}

\medskip
\begin{defn}
Let $\psi: M \to N$ be a morphism of supermanifolds such that $\psi(m) = n$ for some point $m \in M$. The {\it tangent map} of $\psi$ at $m$ is the morphism of $\mathbb{Z}^n_2$-super vector spaces $d\psi_m: T_m M \to T_nN$ defined by
\[
d\psi_m(v)([f]_n) = v(\psi^*_m([f]_n))\;,
\]
for all $v\in T_mM$ and $[f]_n\in \cO_{N,n}$.
\end{defn}
It follows directly from this definition that
\begin{prop}\label{prop:tgtcomp}
Let $\psi: M\to N$ and $\phi:N\to S$ two morphisms of $\Zn$-supermanifolds. Then, for any point $m\in |M|$,
\be\label{comptgtmap}
d(\phi\circ\psi)_m=d\phi_{|\psi|(m)} \circ d\psi_m \;.
\ee
\end{prop}

%------------------------------------------------
\subsection{Chain rule and the modified Jacobian.}
%------------------------------------------------

We now establish the chain rule for $\Zn$-supermanifolds and use it to relate the tangent map to the graded Jacobian matrix.

\begin{prop}\label{chainrule}
Let $U^{p|{\bf q}}, V^{r|{\bf s}}$ be $\mathbb{Z}^n_2$-superdomains, with coordinates
$u^a, v^b$ respectively. %regardless of their degree.
Let $\psi: U^{p|{\bf q}} \to V^{r|{\bf s}}$ be a morphism of supermanifolds. Then,
\[
\frac {\partial \psi^*(f)} {\partial u^a} = \sum_b \frac {\partial \psi^*(v^b)} {\partial u^a} \, \psi^* \bigg( \frac {\partial f} {\partial v^b} \bigg)
\]
for any $f \in \cO(V^{r|{\bf s}})$.
\end{prop}

\begin{proof}
Let $f \in \cO(V^{r|{\bf s}})$,
\beas
D := \frac {\partial \psi^*(f)} {\partial u^a}\;,
&& D' := \sum_b \frac {\partial \psi^*(v^b)} {\partial u^a} \, \psi^* \bigg( \frac {\partial f} {\partial v^b} \bigg)\;,
\eeas
and consider the graded derivation $D - D'$. Clearly, $D - D'$ annihilates all polynomials in the coordinate functions $v^b$, as it is a graded derivation. By Proposition \ref{dercont}, $D - D'$ annihilates all polynomial sections on $V^{r|{\bf s}}$. Let $f \in \cO(V^{r|{\bf s}})$ and $m$ be any point in $V^{r|{\bf s}}$. By polynomial approximation \cite[Thm. 6.10]{CGPa}, for any $k$ there exists a polynomial section $Q$ such that $[f]_m - [Q]_m \in \mathfrak{m}^k_m$. Since this is true for any $k$, we see that $[(D - D')f]_m = 0$, and since $m \in V^{r|{\bf s}}$ was arbitrary, we conclude $(D - D')f = 0$.
\end{proof}

In particular, this implies that
\beas
\left( d\psi_m\;(\partial_{u^a}\vert_m)\right) %([t]_{|\psi|(m)})
&=&
\sum_{b} \left(\partial_{u^a}\psi^*(v^b)\vert_m\right) \cdot \partial_{v^b}\vert_{|\psi|(m)}%\left([t]_{|\psi|(m)}\right)
\eeas
so that the tangent map of $\psi$ at point $M$, is the linear map given by the \emph{block-diagonal} graded matrix
$$
 I_{uv} := \left( \partial_{u^a}\psi^*(v^b)|_m \right)_{ab}\;.
$$
%also called the \emph{Jacobian} of $\psi$ \emph{at $m$}.
Note that, as used in supergeometry, we may simply write $v=v(u)$ instead of $\psi^*(v)$.

\medskip
When considering the counterpart of \eqref{comptgtmap} in matrix form, however, we remark that an additional sign appears. Indeed, for $\Zn$-morphisms $\psi:M\to N$ and $\phi:N\to S$ which read locally respectively as $v=v(u)$ around $m$ and $w=w(v)$ around $|\psi|(m)$, one has
$$
\partial_{u^b}w^a= \sum_c \partial_{u^b}v^c \, \partial_{v^c}w^a
= \sum_c  (-1)^{\langle \deg(u^b)+\deg(v^c),\,\deg(v^c)+\deg(w^a)\rangle} \partial_{v^c}w^a \, \partial_{u^b}v^c \;.
$$
To absorb the redundant sign, we must thus consider the following instead.
\begin{defn}
The \emph{graded Jacobian matrix} of a local $\Zn$-morphisms $\psi$ between $\Zn$-domains $U$ and $V$, given by $v=v(u)$, is the degree $0$ graded matrix
$$
\op{Jac}\psi= \left( (-1)^{\langle \deg(v^i)+\deg(u^j),\,\deg(v^i) \rangle} \partial_{u^j}v^i\right)_{ij}\;.
$$
\end{defn}
\noindent This definition is natural, as the graded Jacobian simply corresponds to the \emph{graded transpose} of the matrix $(\partial_u v)$ (see \cite{Cov}).

We then have the following familiar result.
\begin{prop}
The graded Jacobian matrix of a composition of morphisms corresponds to the graded Jacobians of the individual morphisms, i.e.
\[
J_{\psi \circ \phi} = J_\psi \cdot J_\phi\;.
\]
\end{prop}

\begin{rema}\label{rem:tgmap-Jac}
It is important for the next sections to stress the correspondence between the graded Jacobian and the tangent map of a $\Zn$-morphism.
The tangent map at a point is represented by a $0$-degree graded matrix with entries in $\R$, hence \textbf{block-diagonal}, whereas the modified Jacobian is a ``full'' graded matrix. By definition it is easily seen that the graded matrix corresponding to $d\phi_m$ is the reduced Jacobian matrix evaluated at the point $m$.
More explicitly, writing coordinates $u=(x,\xi_{\gamma}), v=(y,\eta_{\gamma})$ with respect to their $\Zn$-degree, one has
\be\label{tgmapmtrx}
d\phi_m= \op{Jac}_{\phi}\big|_m =
		\left(
			\begin{array}{c|c|c|c}
					B_0 & & & \\
					\hline
					& B_{\gamma_1} & & \\
					\hline
					& & \ddots & \\
					\hline
					& & & B_{\gamma_N}
			\end{array}					
		\right)
\ee
where
\beas
B_0:=\left({\frac{\partial y}{\partial x}}\Big|_m\right)
& \mbox{ and } &
B_{\mu}:= \left({\frac{\partial \eta_{\mu}}{\partial \xi_{\mu}}}\Big|_m\right)
\eeas
for $\mu\in\Zn\setminus\{0\}$, following the standard order.
\end{rema}

%-%-%-%-%-%-%-%-%-%-%-%-%-%-%-%-%-%-%-%
\section{Products of $\Zn$-supermanifolds}
%-%-%-%-%-%-%-%-%-%-%-%-%-%-%-%-%-%-%-%

We begin by stating standard lemmas about gluing together $\Zn$-supermanifolds and morphisms.

\begin{lem}\label{gluesmflds}
Let $X$ be a second-countable Hausdorff topological space, and let $\{U_i\}_{i \in I}$ be a collection of $\Zn$-superdomains such that the $|U_i|$ form an open cover of $X$. Let $U_{ij}$ be the open subdomain of $U_i$ such that $|U_{ij}| = |U_i| \cap |U_j|$. Suppose there exists a collection of $\Zn$-supermanifold isomorphisms $\phi_{ij}: U_{ij} \to U_{ji}$ such that 

\begin{itemize}
\item $\phi_{ii} = id_{U_{ii}}$ for all $i$;
\item $\phi_{ij} \circ \phi_{ji} = id_{U_{ij}}$ for all $i, j$;
\item $|\phi_{ij}|(|U_{ij} \cap U_{ik}|) = |U_{ji} \cap U_{jk}|$, and $\phi_{ik} = \phi_{jk} \circ \phi_{ij}$ on $U_{ij} \cap U_{ik}$ for all $i, j, k$.\\
\end{itemize}

\noindent Then there exists a $\Zn$-supermanifold $M$ with reduced space $X$ and morphisms $f_i: U_i \to M$ for each $i$, uniquely determined up to isomorphism, which are compatible with the inclusions $|U_i| \hookrightarrow X$ and such that $f_j \circ \phi_{ij} = f_i$ for any $i, j$.
\end{lem}

\begin{lem}\label{gluemorph}
If $M, N$ are $\Zn$-supermanifolds, giving a morphism $f: M \to N$ is equivalent to giving an open cover $U_i$ of $M$ and morphisms $f_i$ such that $f_i|_{U_i \cap U_j} = f_j|_{U_i \cap U_j}$ for all $i, j$.  
\end{lem}

Now we will show that the category of $\Zn$-supermanifolds admits finite products.

\begin{prop}
Let $M_i$ be $\Zn$-supermanifolds, $i = 1, 2$. Then there exists a $\Zn$-supermanifold $M_1 \times M_2$, unique up to unique isomorphism, with morphisms $\pi_i: M_1 \times M_2 \to M_i$, such that for any $\Zn$-supermanifold $N$ and morphisms $f_i: N \to M_i$, $i = 1, 2$ there exists a unique morphism $h: N \to M_1 \times M_2$ making the diagram:

\begin{equation*}
\begin{tikzcd}
N \arrow[bend left]{drr}{f_1} 
\arrow[bend right]{ddr}[swap]{f_2} 
\arrow[dotted]{dr}{h} & &\\
& M_1 \times M_2 \arrow{r}{\pi_1} \arrow{d}{\pi_2} & M_1\\
&M_2
\end{tikzcd}
\end{equation*}

\noindent commute.
\end{prop}

\begin{proof}
We begin by noting that if such a product exists, it is uniquely determined up to unique isomorphism by the universal property it is required to satisfy; we will use this uniqueness repeatedly in the proof. We shall construct the product $\Zn$-supermanifold $M \times N$ in steps.\\

\noindent {\bf Step 1.} For the case where $M_1 = \R^{p_1|{\bf q}_1}$, $M_2 = \R^{p_2|{\bf q}_2}$, we show that $M_1 \times M_2$ is $\R^{p_1 + p_2|{\bf q_1} + {\bf q}_2}$, with the morphisms $\pi_i$ given by the corresponding projections onto the factors $\R^{p_i|{\bf q}_i}$.

The Chart Theorem tells us that morphisms $N \to \R^{p_1 + p_2|{\bf q}_1 + {\bf q}_2}$ are in canonical bijection with the set of ordered $p_1 + p_2 |{\bf q}_1 + {\bf q}_2$-tuples of global sections of $\cO_N$. But such an ordered tuple is the same thing as an ordered pair consisting of an ordered $p_1|{\bf q_1}$-tuple of global sections of $\cO_N$, and an ordered $p_2|{\bf q_2}$-tuple of global sections of $\cO_N$. Since ordered $p_1|{\bf q_1}$- (resp. ordered $p_2|{\bf q_2}$-) tuples of global sections of $\cO_N$ are in canonical bijection with morphisms $N \to M_1$ (resp. morphisms $N \to M_2$), we see that to any pair of morphisms $f_i: N \to \R^{p_i|{\bf q}_i}$, there is canonically associated a unique morphism $h: N \to \R^{p_1 + p_2|{\bf q}_1 + {\bf q}_2}$. That this $h$ makes the stated diagram commute is straightforward to prove.\\

\noindent {\bf Step 2.} Now suppose $M_i$ are $\Zn$-supermanifolds such that the product $M_1 \times M_2$ exists, and let $U$ be an open subsupermanifold of $M_1$. Then $\pi_1^{-1}(U)$ is the product of $U$ and $M_2$. To see this, let $f_1: N \to U, f_2: N \to M_2$ be morphisms. Then by composing with the inclusion $U \to M_1$, we have a morphism $\widetilde{f}_1: N \to M_1$. By the universal property of the product $M_1 \times M_2$, there is a unique morphism $h: N \to M_1 \times M_2$ which is compatible with $\widetilde{f}_1, f_2$ and the projections to the $M_i$. However, as $|\widetilde{f}_1|(|N|) \subseteq |U|$, we have $|h|(|N|) \subseteq \pi^{-1}_1(U)$, whence we may regard $h$ as a morphism $N \to \pi^{-1}(U)$. $h$ is readily seen to be the unique morphism compatible with the $f_i$ and the projections to $U$ and $M_2$, hence $\pi^{-1}_1(U)$ is the product $U \times M_2$. The same argument works to show that if $V$ is an open subsupermanifold of $M_2$, $\pi^{-1}_2(V)$ is the product $M_1 \times V$.\\

\noindent {\bf Step 3.} Let $V$ be a $\Zn$-superdomain in $\R^{p_2|{\bf q}_2}$. We prove existence of the product $M_1 \times V$, for $M_1$ an arbitrary $\Zn$-supermanifold. Fix an open cover $\{U_i\}_{i \in I}$ of $M_1$ by $\Zn$-superdomains. The products $U_i \times V$ exist by steps 1 and 2. Let $U_{ij} = U_i \cap U_j$. Then by step 2, $W_{ij} := \pi_1^{-1}(U_{ij})$ are the products $U_{ij} \times V$. (We note that the $U_i \times V$, hence the $W_{ij}$, are $\Zn$-superdomains). Uniqueness of products implies that there are unique isomorphisms $\phi_{ij}: W_{ij} \to W_{ji}$ for each $i, j$ which are compatible with the projection morphisms to $U_{ij}$ and $V$.

One checks that the $\phi_{ij}$ satisfy the hypotheses of Lem. \ref{gluesmflds}, hence the $U_i \times V$ may be glued together via the isomorphisms $\phi_{ij}$, yielding a $\Zn$-supermanifold which we denote by $M_1 \times V$. ({\it A priori}, this is an abuse of notation, but we will show that $M_1 \times V$ so defined indeed satisfies the universal property characterizing the product of $M_1$ and $V$). By Lem. \ref{gluemorph}, the projections $(\pi_1)_i, (\pi_2)_i$ to $U_i$ and $V$ glue together to form morphisms $\pi_1, \pi_2$ to $M_1$ and $V$. Now let $f: N \to M_1$, $g: N \to V$ be morphisms, and let $N_i := f^{-1}(U_i)$. By Step 2, this gives unique morphisms $h_i: N_i \to U_i \times V$ compatible with $f|_{N_i}, g|_{N_i}$ and the projections $(\pi_k)_i$ from $U_i \times V$ to $U_i$ and $V$; composing the $h_i$ with the inclusions $U_i \times V \to M_1 \times V$ gives morphisms from $N_i \to M_1 \times V$ which we continue to denote by $h_i$. One checks that $h_i$ and $h_j$ agree on $N_i \cap N_j$ for any $i,j$, hence by Lem. \ref{gluemorph} they may be glued into a morphism $h: N \to M_1 \times V$. One sees that $h$ commutes with $f, g$ and the projections $\pi_k$ to $M_1$ and $V$ from the corresponding properties of the $h_i$. The uniqueness of $h$ follows from the uniqueness of the $h_i$. Hence we have shown that $M_1 \times V$ is the product of $M_1$ and $V$.\\ 

\noindent {\bf Step 4.} Once we know the existence of the product $U \times M_2$ for $U$ an open subsupermanifold $U$ of $\R^{p_1|{\bf q}_1}$, the same argument we just gave to prove the existence of the product $U_1 \times M_2$ goes through to prove the existence of the product $M_1 \times M_2$. The details are left to the reader.

\end{proof}

It is clear that one can iterate this construction to obtain the product $M_1 \times M_2 \times \dotsc \times M_k$ of any finite family of supermanifolds $M_i$, $i = 1, \dotsc, k$, so that the category $\Zn\text{-}\mathtt{Man}$ admits arbitrary finite products.

%-%-%-%-%-%-%-%-%-%-%-%-%-%-%
\section{Local forms of morphisms}
%-%-%-%-%-%-%-%-%-%-%-%-%-%-%

%------------------------------------------------
\subsection{Inverse function theorem} \label{ssec:IFT}
%------------------------------------------------

We begin with the analogue of the inverse function theorem, which plays as fundamental a role in $\Zn$-supergeometry as the classical version does in ungraded geometry.

\begin{thm}\label{InvFunThm}
Let $\phi: M \to N$ be a morphism of $\Zn$-supermanifolds, and $m \in |M|$ be a point. Then the following are equivalent.
\begin{enumerate}
\item $d\phi_m$ is invertible;
\item there exist coordinate charts $U$ about $m$, $V$ about $\phi(m)$ such that $\phi|_U: U \to V$ is a diffeomorphism.
\end{enumerate}
\end{thm}

\begin{proof}
That $\phi$ being a diffeomorphism around $m$ implies invertibility of $d\phi_m$ is an easy consequence of the chain rule.

Now let us suppose that $d\phi_m$ is invertible; in particular, this implies $\dim(M) = \dim(N)$. As the statement is local, we may assume from the beginning that $M$ and $N$ are $\Zn$-superdomains $U^{p|{\bf q}}$ and $V^{p|{\bf q}}$ respectively, and that $m = 0$. Let $\mu \in \Zn \backslash \{0\}$. Let $(x, \xi)$ and $(y, \eta)$ be coordinates on $U, V$ respectively. In order to keep track of the $\Zn$-degrees of the variables, we have introduced in Section \ref{sec:prelim} the notation $\xi^j_\mu$ to indicate a coordinate of $\Zn$-degree $\mu$; similarly for the $\eta^s_{\mu}$.

By the Chart Theorem for $\Zn$-supermanifolds \cite[Thm. 6.8]{CGPa}, we have:
\begin{align*}
&\phi^*(y^r) = f_0^r(x)+\sum_{\substack{|P| > 0\\ \deg(\xi^P)=0}} f_P^r(x)\, \xi^P\\
&\phi^*(\eta^s_\mu) = \sum_{\substack{|Q| \geq 1 \\ \deg(\xi^Q)=\mu}} g_{Q}^s(x)\, \xi^Q
										= \sum_{k=1}^{q_{\mu}} g_{\mathcs{E}(\mu,k)}^s(x) \,\xi^k_{\mu} + \ldots
\end{align*}
where, for $\mu$ the $i$-th non-zero degree in $\Zn$ following the standard order, $\mathsc{E}(\mu,k)=(0,\ldots, 1, \ldots, 0)$ with $1$ at the entry $l:= q_1+ \ldots + q_{i-1}+k$.
Then, recalling the matrix form of the differential of $\phi$ at $m$ \eqref{tgmapmtrx}, the hypothesis that the differential be invertible is equivalent to the assumption that the block matrices

\begin{align*}
B_0
= \begin{pmatrix}
\frac {\partial f^1_0} {\partial x^1} & \dotsc & \frac {\partial f^1_0} {\partial x^p} \\
\vdots & & \vdots\\
\frac {\partial f^p_0} {\partial x^1} & \dotsc & \frac {\partial f^p_0} {\partial x^p}
\end{pmatrix}
\hspace{7mm}
B_\mu = \begin{pmatrix}
g^1_{\mathsc{E}(\mu,1)} & \dotsc & g^1_{\mathsc{E}(\mu, q_k)}  \\
\vdots & & \vdots\\
g^{q_\mu}_{\mathsc{E}(\mu,1)} & \dotsc & g^{q_\mu}_{\mathsc{E}(\mu, q_k)}
\end{pmatrix}
\end{align*}\\

\noindent are invertible for all $\mu \in \Zn \backslash \{0\}$. By the classical inverse function theorem, there exists an open subset $W^{p|{\bf q}}$ such that the $\{ f^i_0 \}$ define a new coordinates of degree $0$, denoted by $\widetilde{x}$. Furthermore, using the invertible matrices $B_{\mu}$ we can also change the coordinates of each non-zero degree $\mu$, by
$$
\widetilde{\xi}^a_{\mu}:= \sum_{b} g^a_{\mathsc{E}(\mu,b)}(x) \,\xi^b_{\mu}
$$
finally obtaining a new system of coordinates $(\widetilde{x}, \widetilde{\xi}\,)$ near the considered point.

Denoting this change of coordinates by $\tau$, we have:
\begin{align*}
&(\tau^* \circ \phi^*)(y^j) = \widetilde{x}^j + \sum_{|P| \geq 2} \widetilde{f}^j_P(\widetilde{x})\,\widetilde{\xi}^P\\
&(\tau^* \circ \phi^*)(\eta^s_\mu) = \widetilde{\xi}^s_\mu + \sum_{|Q| \geq 2} \widetilde{g}^j_{Q} (\widetilde{x})\,\widetilde{\xi}^Q
\end{align*}
We see from this that $\tau^* \circ \phi^*$ is of the form $\I + J$, where $\I$ is the identity matrix and $J$ is a matrix whose entries lie in $\mathcal{J}$. Hence $\tau^* \circ \phi^*$ has a formal inverse $\psi$ given by the geometric series $\sum_k^\infty (-1)^k J^k$, which converges to a unique matrix by $\cJ$-adic Hausdorff completeness of $\mathcal{O}$. Thus $\phi^*$ is actually invertible, and the Chart Theorem implies that $\phi$ is a diffeomorphism.
\end{proof}

%------------------------------------------------
\subsection{Immersions and submersions}
%------------------------------------------------

\begin{defn}
A morphism $\phi:M\to N$ of $\Zn$-supermanifolds is an \emph{immersion} (resp. a \emph{submersion}) at a point $m\in |M|$ if its tangent map at this point $d\phi_m$ is injective (resp. surjective).
\end{defn}
Since it is a local property, we may for simplicity restrict ourselves in this section to considering $\Zn$-superdomains instead of general $\Zn$-supermanifolds.

Let us consider a morphism $\phi: U \to V$ between $\Zn$-superdomains $U^{p|\mathbf{q}}$ and $V^{r|\mathbf{s}}$, with coordinate systems $u=(x,\xi_{\mu})$ and $v=(y,\theta_{\mu})$, respectively.  %sub-$\Zn$-domains of $M$ and $N$ around $m=u_0$ and $|\phi|(m)$.
%relation with rank of smatrx
At a point $m\in U$, the differential $d\phi_{m}$ is a linear map between $\Zn$-graded vector spaces, which is represented by a block-diagonal matrix \eqref{tgmapmtrx} in $\gl^0(r|\mathbf{s}\times p|\mathbf{q}; \R)$.
The rank of such a graded matrix is intuitively given by the ranks of its diagonal blocks (which are classical ungraded real matrices). Hence,

\begin{itemize}
\item $\phi$ is an immersion at $m$ if and only if $\op{rank}(d\phi_{m})=p|\mathbf{q}$\;;
\item $\phi$ is an submersion at $m$ if and only if $\op{rank}(d\phi_{m})=r|\mathbf{s}$\;.
\end{itemize}

We have moreover the following results.
\begin{prop}\label{im-sub}
Let $\phi:M\to N$ be a morphism of $\Zn$-supermanifolds, $M$ of dimensions $p|\mathbf{q}$ and $N$ of dimension $r|\mathbf{s}$, and $m$ a point of $M$.
	\begin{enumerate}
			\item\label{immersion} Let $p\leq r$ and $q_j\leq s_j$, for all  $j$. The map $\phi$ is an immersion at $m$ if and only if there exists $\Zn$-supercharts $(U,(x,\xi_{\mu}))$ around $m$ and $(V,(y,\theta_{\mu}))$ around $|\phi|(m)\in |N|$ on which it has the form
			\beas
			\phi^*_V(y^i)=\begin{cases}
														x^i &\mbox{ if } 1\leq i\leq p \\
														0 &\mbox{ if } p+1\leq i\leq r
										\end{cases}
			& \mbox{ and } &
			\phi^*_V(\theta_{\mu}^{a})=\begin{cases}
																					\xi_{\mu}^{a} &\mbox{ if } 1\leq a\leq q_{\mu} \\
																					0 &\mbox{ if } q_{\mu}+1\leq a\leq s_{\mu}
																				\end{cases}
			\eeas
			for all $\mu\in\Zn\setminus\{0\}$. Here we denote by $q_{\mu}$ the entry of $\mathbf{q}$ corresponding to the degree $\mu$, and analogously for $s_{\mu}$.
			
			\item\label{submersion} Let $p\geq r$ and $q_j\geq s_j$, for all $j$. The map $\phi$ is a submersion at $m$ if and only if there exists $\Zn$-supercharts $(U,(x,\xi_{\mu}))$ around $m$ and $(V,(y,\theta_{\mu}))$ around $|\phi|(m)\in |N|$ on which it has the form
			\bea\label{submdef}
			\phi^*_V(y^i)= x^i \,;\quad 1\leq i\leq r
			& \mbox{ and } &
			\phi^*_V(\theta_{\mu}^{a})= \xi^{a}_{\mu}\;, \quad 1\leq a\leq s_{\mu}\;,
			\eea
			for all $\mu\in\Zn\setminus\{0\}$.
	\end{enumerate}
\end{prop}
In other words, $\phi$ is an immersion (resp. a submersion) at $m$ if and only if in a neighborhood of $m$ and $\phi(m)$ there exist coordinates in which $\phi$ is a linear injection (resp. a linear projection).

\begin{proof}
As already mentioned, without loss of generality we may restrict ourselves to the case of a morphism between $\Zn$-superdomains $\phi:U\to V$, using the above notations.

The proofs of the two claims \eqref{immersion} and \eqref{submersion} are analogous, and follow without major modifications the reasoning used to prove the super case (see e.g. \cite{Lei}).
%In this latter case, the proof of the immersion property is described, leaving the submersion property to the reader.
For the sake of completeness, we will describe here the proof for the submersion property \eqref{submersion}, which is usually left in the super case in favor of the one for the immersion property.

First, given a morphism $\phi:U\to V$ defined in a neighborhood of $m$ by \eqref{submdef}, one directly verify that $d\phi_{m}$ is of the form
$$
d\phi_{u_0}= \left(\begin{array}{cc|cc|cc|cc}
								\mathbb{I}& 0 & & & & & &\\
								\hline
								& &\mathbb{I}&0  & & & &\\
								\hline
								& & & & \ddots & & &\\
								\hline
								& & & & & &\mathbb{I}& 0
						\end{array}\right)
						\in \gl^0(r|\mathbf{s}\times p|\mathbf{q};A)\;.
$$
which is of rank $r|\mathbf{s}$. Thus, $\phi$ is a submersion at $m$.

For the converse, up to reordering of the coordinates we can assume that the sub-blocks

\begin{align*}
&\left(\frac{\partial \phi^*(y^k)}{\partial x^i}\Big|_{m}\right)_{i,k=1, \ldots,r} \mbox{ of the diagonal block $B_0$ } \\
\mbox{and } & \left(\frac{\partial \phi^*(\theta_{\mu}^a)}{\partial \xi_{\mu}^b}\Big|_{m}\right)_{a,b=1, \ldots,s_{\mu}} \mbox{ of the diagonal blocks $B_{\mu}$}\;,
\end{align*}

are invertible matrices.

Let us consider the $\Zn$-superdomain $\R^{p-m|\mathbf{q}-\mathbf{s}}$ and relabel its coordinate as $(v^{r+i}, \theta_{\mu}^{s_{\mu}+a_{\mu}})$, $1\leq i \leq p-m$ and $1\leq a_{\mu}\leq q_{\mu}-s_{\mu}$. Thanks to the fundamental theorem of $\Zn$-morphisms \cite[Theorem 6.8]{CGPa}, we can construct a morphism $\psi:U'\to V\times \R^{p-m|\mathbf{q}-\mathbf{s}}$ in a neighborhood $U'$ of $m$ by setting
$$
\psi^*(v^i)= \begin{cases}
								\phi^*(v^i) & \mbox{ if } 1\leq i\leq r\\
								u^i & \mbox{ if } r+1\leq i\leq p
						\end{cases}
$$
and, for all $\mu\in\Zn\setminus\{0\}$,
$$
\psi^*(\theta_{\mu}^{a})= \begin{cases}
																\phi^*(\theta_{\mu}^{a}) & \mbox{ if } 1\leq a\leq s_{\mu}\\
																\xi_{\mu}^{a} & \mbox{ if } s_{\mu}+1\leq a\leq q_{\mu}
															\end{cases}
\;.$$
By construction, we then have that in a neighborhood of $m$ the original morphism may be written as $\phi=\op{proj}_V\circ \psi$ where $\op{proj}_V:V\times \R^{p-m|\mathbf{q}-\mathbf{s}} \to V$ is the natural projection.
Moreover, it is straightforward to check that $d\psi_{m}$ is invertible, and so by Theorem \ref{InvFunThm} $\psi$ is a local diffeomorphism in a neighborhood of $m$.
Hence, the claim \eqref{submersion} follows.
\end{proof}

%------------------------------------------------
\subsection{The constant rank theorem}
%------------------------------------------------
\begin{defn}
Let $A$ be a $J$-adically Hausdorff complete $\Zn$-supercommutative ring, and let $Z \in \gl^0({ m|{\bf n}\times p|{\bf q} };A)$. The graded matrix $Z$ has {\it constant rank} $r|{\bf s}$ if there exist $G_1 \in \GL(m|{\bf n}; A)$ and $G_2 \in \GL(p|{\bf q}; A)$ such that 

\be\label{cstrank}
G_1 Z G_2 = 
\left(\begin{array}{c|c|c|c}
\begin{array}{cc}
\mathbb{I}_r & 0\\
0 & 0
\end{array} 
									&    &  &    \\
\hline
 &  \begin{array}{cc}
		\mathbb{I}_{s_1} & 0\\
			0 & 0
		\end{array}  						&    &  \\
\hline
 &  & \ddots  &  \\
\hline
  &   &  & \begin{array}{cc}
\mathbb{I}_{s_N} & 0\\
0 & 0
\end{array} \\
\end{array}\right).
\ee
\end{defn}

It is straightforward to see that if a matrix $Z$ of graded functions on a $\Zn$-superdomain $U$ has constant rank $r|\mathbf{s}$, then this matrix evaluated at a point $m\in |U|$ is a block-diagonal matrix $Z|_m$ of rank $r|\mathbf{s}$. As in the super case, the converse is not true.

\begin{thm}[Constant rank theorem on $\Zn$-supermanifolds]
Let $\phi:M\to M'$ be a morphism of $\Zn$-supermanifolds, of respective dimensions $p|\mathbf{q}$ and $p'|\mathbf{q'}$, $m\in |M|$. Then the following are equivalent.
\begin{enumerate}
	\item In a neighborhood of $m$, $\op{Jac}_{\phi}$ is a graded matrix of constant rank $r|\mathbf{s}$;
	\item In a neighborhood $U$ of $m$, $\phi$ may be written as the composite of a submersion $\phi_1:U\to W$  at $m$ and an immersion $\phi_2:W\to V$ at $|\phi_1|(m)$.
\end{enumerate}
\end{thm}

\begin{proof}
That (2) implies (1) clearly follows from %\eqref{comptgtmap} and from the definitions of sub- and im-mersions. 
the property of sub- and im-mersions (Proposition \ref{im-sub}).

The converse is more involved. Since the statement is local, as before we may restrict ourselves to consider a morphism of $\Zn$-superdomains $\phi:U\to V$, with coordinates $u=(x,\xi_{\mu})$ %near $m\in |U|$ 
and $v=(y,\eta_{\mu})$.% near $m':=|\phi|(m)\in |V|$. 

Let us assume that $\op{Jac}_{\phi}$ has constant rank $r|\mathbf{s}$ near the point $m$. By the above observation, the block-diagonal matrix $d\phi_m=\op{Jac}_{\phi}|_m$ is of rank $r|\mathbf{s}$. Up to reordering of the coordinates, we may thus assume that the sub-blocks of $d\phi_m$ that are invertible are 
\bea\label{diff}
 \left(\frac{\partial \phi^*(y^i)}{\partial x^j}\right)_{\substack{1\leq i,j \leq r\\ }} 
& \mbox{ and }&
\left(\frac{\partial \phi^*(\eta_{\mu}^i)}{\partial \xi_{\mu}^j}\right)_{\substack{1\leq i,j \leq s_{\mu}\\ }}
\eea
for all $\mu\in \Zn\setminus\{0\}$.

\medskip
The strategy of the proof is to construct a submersion and an immersion out of $\phi$, and then prove equality between their composite and the original morphism. 
\medskip
Let $W:=\R^{r|\mathbf{s}}$ a $\Zn$-domain with coordinates $w=(z,\zeta_{\mu})$. We define a morphism of $\Zn$-domains $\phi_1:U\to W$ by setting 
\beas
\phi_1^*(z^i)= \phi^*(y^i) & \mbox{ and } & \phi_1^*(\zeta^a_{\mu})= \phi^*(\eta^a_{\mu})
\eeas
%($1\leq i \leq r$, $1\leq a \leq s_{\mu}$), 
for all $\mu\in\Zn\setminus\{0\}$. By construction and \eqref{diff}, the differential of $\phi_1$ at $m$ has rank $r|\mathbf{s}$, and so $\phi_1$ is a submersion at this point. 
In particular, by the second point of Proposition \ref{im-sub}, up to a change of the first $r|\mathbf{s}$ coordinates on $U$ near $m$, the morphism $\phi_1$ (and thus $\phi$) writes as 
\be\label{newform}
\phi_1^*(z^i)=x^i \mbox{  and  } \phi_1^*(\zeta^a_{\mu})=\xi^a_{\mu},
\ee 
for $1\leq i \leq r$, $1\leq a \leq s_{\mu}$. 

\medskip
Let us now consider the morphism of $\Zn$-domains $\psi:W\to U$ defined by 
\beas
\psi^*(x^j)=\begin{cases}
								z^j & \mbox{ for } 1\leq j \leq r\\
								0 & \mbox{ for } r+1\leq j \leq p
						\end{cases}
& \mbox{ and } & 
\psi^*(\xi_{\mu}^b)=\begin{cases}
								\zeta_{\mu}^b & \mbox{ for } 1\leq b \leq s_{\mu}\\
								0 & \mbox{ for } s_{\mu}+1\leq b \leq q_{\mu}
						\end{cases}
\eeas
Then, using \eqref{newform}, one easily sees that 

\be\label{newform2}
\begin{aligned}
&(\phi\circ\psi)^*(y^i):=\psi^*\circ \phi^*(y^i) = z^i & \mbox{ for all } 1\leq i \leq r \\ 
&(\phi\circ\psi)^*(\eta_{\mu}^a) = \zeta_{\mu}^a & \mbox{ for all } 1\leq a \leq s_{\mu} \;.
\end{aligned}
\ee

Hence, the composite $\phi_2:=\phi\circ \psi: W\to V $ is a immersion. By the first point of Proposition \ref{im-sub}, up to change of the last $p'-r|\mathbf{q'}-\mathbf{s}$ coordinates around $|\phi|(m)$ in $V$, the morphism $\phi_2$ writes in particular as 
\bea\label{newform3}
\phi_2^*(y^j)=0 &\mbox{  and  }& \phi_2^*(\eta^b_{\mu})=0,
\eea
for $r+1\leq j \leq p'$ and  $s_{\mu}+1\leq b \leq q'_{\mu}$. 

\medskip
It remains now to show equality between the original morphism $\phi$ and the composite $\phi':=\phi_1\circ \phi_2$. By construction,  $\phi$ and $\phi'$ coincide when restricted to the subdomain of $U$ defined by equations $\{x^j=0\,,\, r+1\leq j\leq p\}$, and $\{\xi^b_{\mu}=0\,,\, s_{\mu}+1\leq b\leq q_{\mu}\}$ for all $\mu\in\Zn\setminus\{0\}$ (see Remark \ref{rem:subdomain}).    

On the other hand, by the argument pertaining to \eqref{newform}, considered on a sufficiently small neighborhood of $m$, the morphism $\phi$ may be written in particular as 
\beas
\phi^*(y^i)= x^i  \mbox{ for } 1\leq i\leq r
 & \mbox{ and }&
\phi^*(\eta_{\mu}^a)= \xi_{\mu}^a  \mbox{ for } 1\leq a\leq s_{\mu}
\eeas
after a change of the first $r|\mathbf{s}$ coordinates. 
The graded Jacobian of $\phi$ near $m$ is hence a graded matrix of constant rank $r|\vect{s}$ of the form 
$$
J=
\left(
\begin{array}{cc|cc|cc|cc}
\mathbb{I}_r & 0 & 0 & 0 & \ldots & &0 & 0 \\
\star & \bullet & \star & \bullet & \ldots & &\star &\bullet \\
\hline
0 & 0 &\mathbb{I}_{s_1} & 0 &\ldots & &0 & 0 \\
\star & \bullet & \star & \bullet & \ldots & &\star &\bullet \\
\hline
\vdots & \vdots &  &  &\ddots & & \vdots& \vdots \\
 %&  &  &  &  & & & \\
\hline
0 & 0 &0 & 0 &\ldots &  &\mathbb{I}_{s_{N}} & 0\\
\star & \bullet & \star & \bullet &  & &\star &\bullet \\
\end{array}
\right)
\in \gl^0(p'|\mathbf{q'} \times p|\mathbf{q};\cO(U))
$$
%%%%%%commented
\begin{comment}
which we can re-write regrouping the first $r|\mathbf{s}$ rows and columns 
\be\label{mtrxphi}
J:=\left(
\begin{array}{ccc|ccc}
& & & & &  \\
& \mathbb{I}_{r|\mathbf{s}}& & & 0& \\
& & & & & \\
\hline
\star & \ldots & \star & \bullet & \ldots & \bullet\\
 \vdots & \ddots & \vdots &  \vdots & \ddots & \vdots\\
 \star & \ldots & \star& \bullet & \ldots & \bullet
\end{array}
\right)
\ee
Note that in this way, we do not have a graded matrix as before, but, intuitively, a block matrix in which each block is a graded matrix. In fact, such a matrix can still be considered a graded matrix of a different type, in non-block form (see \cite{CM}), hence the usual matrix operations still make sense.
\end{comment}

Since $J$ is by hypothesis of constant rank $r|\mathbf{s}$, there exist invertible degree-zero matrices $G_1$,$G_2$ of functions over the $\Zn$-domain $U$, such that $G_1JG_2$ is of the form \eqref{cstrank}. In particular, $JG_2=G_1^{-1}(G_1JG_2)$ has exactly $r|\mathbf{s}$ columns which are non-zero. The matrix $G_2$ is hence necessarily of the form 
$$
G_2= 
\left(
\begin{array}{cc|cc|cc|cc}
A_{11} & 0 & A_{12} & 0 & \ldots & &A_{12^n} & 0 \\
C_{11} & D_{11} & C_{12} & D_{12} & \ldots & &C_{12^n} &D_{12^n} \\
\hline
A_{21} & 0 &A_{22} & 0 &\ldots & &A_{22^n} & 0 \\
C_{21} & D_{21} & C_{22} & D_{22} & \ldots & &C_{22^n} &D_{22^n} \\
\hline
\vdots & \vdots &  &  &\ddots & & \vdots& \vdots \\
 %&  &  &  &  & & & \\
\hline
A_{2^n1} & 0 &A_{2^n2} & 0 &\ldots &  &A_{2^n2^n} & 0\\
C_{2^n1} & D_{2^n1} & C_{2^n2} & D_{2^n2} &  & &C_{2^n2^n} &D_{2^n2^n} \\
\end{array}
\right)
$$
with $D_{ii}$ all invertible (by invertibility of $G_2$ and Proposition \ref{invertibleblockdiag}, see Appendix). 
And thus, moreover, %the low diagonal block of $J$ in the form \eqref{mtrxphi} should also be $0$. 
the ``$\bullet$''-subblocks of $J$ are necessarily $0$. 
In other words, 
$$
\begin{aligned}
\frac{\partial \phi^*(y)}{\partial x^j}=0 &\quad \mbox{ for } p'-r\leq j \leq p' \\
\frac{\partial \phi^*(\eta_{\mu})}{\partial \xi_{\nu}^a}=0 &\quad \mbox{ for all $\mu,\nu$ and } q'_{\nu}-s_{\nu}\leq a \leq q'_{\nu}
\end{aligned}
$$

The above equalities clearly hold also for the morphism $\phi':=\phi_1\circ \phi_2$. The claim then follows from Lemma \ref{lem:eqf}.
\end{proof}

\begin{lem}\label{lem:eqf}
Let $U$ be a $\Zn$-domain, with $k+r|\mathbf{l}+\mathbf{s}$ coordinates $u=(x, t,\xi_{\mu},\theta_{\mu})$, and let us denote by $U'$ its $\Zn$-subdomain defined by the equations $t=0$ and $\theta_{\mu}=0$ (for all $\mu\in\Zn\setminus\{0\}$). If two graded functions $f_1,f_2$ on $U$ are such that  $f_1|_{W}=f_2|_{W}$ and 
\be\label{diff0}
\frac{\partial f_i}{\partial t}=0= \frac{\partial f_i}{\partial \theta_{\mu}}\; \qquad (i=1,2),
\ee
then they coincide on some neighbourhood of $U'$. 
\end{lem}

\begin{rema}\label{rem:subdomain}
Let us be more precise regarding the $\Zn$-subdomain $U'$ of $U$. Its base space is the subdomain of $|U|$ defined by equations $v=0$, whereas its structure sheaf consists of functions of the form formal power series in the $\xi_{\mu}$-variables, with coefficients which are smooth functions in the $x$-variables ($\Ci(x)[[\xi]]$). There is a natural embedding $\rho: U'\to U$ given by 
$$
 \rho^*(x)=x, \quad \rho^*(t)=0, \quad \rho^*(\xi)=\xi, \quad \rho^*(\theta)=0 \;.
$$
Then, for any $f\in \Ci(U)$ we denote $f|_W:=\rho^*(f)$. 
\end{rema}

\begin{proof}
Set $f=f_1-f_2\in \Ci(U)$, which has the coordinate expression 
$$  
f=\sum_{\alpha, \beta}^{\infty} f_{\alpha,\beta}(x,t)\xi^{\alpha}\theta^{\beta}\;,
$$
where the coefficients $f_{\alpha,\beta}$ are classical smooth functions on the variables $(x,t)$. 
From \eqref{diff0} (thanks to the $\cJ$-adic continuity of partial derivatives), it follows that $f_{\alpha,\beta}=0$ whenever $\beta\neq 0$. Moreover, also from the hypothesis, for all $\alpha$ we have 
\beas
f_{\alpha,0}|_{W}=0 &\mbox{ and }& \frac{\partial f_{\alpha,0}}{\partial t}=0\;,
\eeas
which implies that $f_{\alpha,0}=0$ in a neighbourhood of $U'$. The claim follows.  
\end{proof}

\bigskip

%-%-%-%-%-%-%-%-%-%-%-%-%-%-%
\section{Cotangent sheaf and differential forms}
%-%-%-%-%-%-%-%-%-%-%-%-%-%-%

\begin{defn}
The {\it cotangent sheaf} of a $\Zn$-supermanifold $M$ is the sheaf of topological $\cO$-modules 
$$
\Om^1_M = \cT^*\!M:= \iHom_{\cO}(\cT M, \cO).
$$ 
\end{defn}

Since $\cT M$ is a locally free sheaf of rank $p|{\bf q}$, one sees from its definition that $\Om^1_M$ is also a locally free sheaf of rank $p|{\bf q}$. In terms of a coordinate system $u=(x^i, \xi_{\mu}^a)$ on a chart $U$, a local $\cO(U)$-basis of $\Om^1_M(U)$ 
%at a point $m \in U$ 
is given by the linear functionals $du=(dx^i, d\xi_{\mu}^a)$
dual to the coordinate vector fields $\partial_{u}=(\partial_{x^i}, \partial_{\xi^j_{\mu}})$.

Following Deligne-Morgan's convention in \cite{DM}, we write the evaluation pairing $( \; , \, ): \cT M \, \otimes \, \cT ^*\!M \to \cO$ by setting 
$$
( fX, g\omega ) = (-1)^{\langle \deg(X), \deg(g) \rangle} fg \,( X, \omega )
$$
for $f,g \in \cO(U)$, $X \in \cT M(U)$, $\omega \in \cT^*\!M(U)$.

\medskip

%------------------------------------------------
\subsection{Differential forms}
%------------------------------------------------

Since the category of graded $A$-modules for a $\Zn$-commutative ring $A$ is a symmetric monoidal category, the exterior powers and the exterior algebra of an $A$-module, as well as arbitrary direct sums of $A$-modules, are well-defined (see the Appendix as well as \cite{CM}, Section 3.1.2 for more details). Hence we may define differential forms for $\Zn$-supermanifolds.

\begin{defn}
Let $M$ be a $\Zn$-supermanifold. The sheaf of {\it differential $k$-forms} $\Om^k_M$ is the \emph{extension of the $\cB$-sheaf} (see \cite[Appendix 7.3.1]{CGPa}) of $\cO$-modules $U \mapsto \Lambda^k_{\cO(U)}(\Om^1(U))$. Analogously, the sheaf of {\it differential forms} on $M$ is the extension of the $\cB$-sheaf $\Om^\bullet_M : U \mapsto \Lambda^\bullet_{\cO(U)}(\Om^1(U)) = \bigoplus_k \Om^k_M(U)$. 
\end{defn}

\begin{prop}
For any $k$, $\Om^k_M$ is $\cJ$-adically Hausdorff complete, and $\Om^k_{M, m}$ is $\cJ_m$-adically Hausdorff complete for any $m \in M$. The same is true of $\Om^\bullet_M$ and $\Om^\bullet_{M, m}$.
\end{prop}

\begin{proof}
This follows from the $\cJ$-adic Hausdorff completeness of $\cO$ and the fact that $\Om^k_M$ is a locally free sheaf; a similar argument holds for $\Om^k_{M,m}$.
\end{proof}

In addition to their $\Zn$-grading, differential forms have a natural $\N$-grading from their definition as a sheaf of exterior algebras (namely, elements in $\Om^k(M)(U)$ have $\N$-degree $k$).

The wedge product on $\Om^\bullet_M$ is just the product on the exterior algebra, so is compatible with both the $\N$-grading and the $\Zn$-grading. With this grading, the wedge product turns $\Om^\bullet(M)$ into a commutative $\N$-graded $\Zn$-superalgebra:
\[
\alpha \wedge \beta = (-1)^{\langle \deg(\alpha), \deg(\beta) \rangle + |\alpha||\beta|} \, \beta \wedge \alpha\;.
\]
Here $| \cdot |$ denotes the $\N$-degree of a differential form. Note that we follow the \emph{Deligne sign rule} for the commutation of differential forms, explained in \cite{DM}.

\begin{rema} As in the case of ordinary supergeometry, in general there are no ``top-degree forms" on a $\Zn$-supermanifold since $\Om^k(M^{p|{\bf q}})$ is nonzero for all $k$ if ${\bf q}$ is not purely even.
\end{rema}

\medskip
The wedge product of differential forms is compatible with the $\cJ$-adic topology:
\begin{prop}
The wedge product of differential forms $\Om^k_M \times \Om^l_M \to \Om^{k+l}_M$ is $\cJ$-adically continuous for any $k, l$. Consequently, the wedge product $\Om^\bullet_M \times \Om^\bullet_M \to \Om^\bullet_M$ is $\cJ$-adically continuous.
\end{prop}

\begin{proof}
It suffices to verify the proposition locally in a coordinate domain $U$ with coordinates $(x^i, \xi^j)$. For the remainder of the proof, we denote the wedge product by $\nu: \Om^k_M(U) \times \Om^l_M(U) \to \Om^{k+l}_M(U)$. A direct calculation in coordinates shows that $\nu((\cJ^m \Om^k)(U) \times (\cJ^n \Om^l)(U)) \subseteq (\cJ^{m+n} \Om^{k+l})(U)$ for any $m$ and $n$. This implies in particular that we have

\begin{align*}
\nu^{-1}(\omega + (\cJ^m \Om^{k+l})(U)) = \bigcup_{(\sigma, \tau) \in \nu^{-1}(\omega + (\cJ^m \Om^{k+l})(U))} \left(\sigma + (\cJ^m \Om^k)(U) \right) \times \left(\tau+ (\cJ^m \Om^l)(U) \right).
\end{align*}

\noindent for any $\omega \in \Omega^{k+l}(U)$. The sets $\left(\sigma + (\cJ^m \Om^k)(U) \right) \times \left(\tau+ (\cJ^m \Om^l)(U) \right)$ are open in the product topology for $\Om^k_M(U) \times \Om^l_M(U)$, whence the LHS of the above equality is also open, proving continuity of $\nu$.

\end{proof}

A similar $\cJ$-adic continuity statement holds for the wedge product on germs of differential forms at a given point $m \in M$, and may be proven in a similar fashion.\\

Finally, we note that given a morphism of $\Zn$-supermanifolds $\Phi: M \to N$, there is a pullback morphism $\Phi^*: \Omega^1(N) \to \Omega^1(M)$, which is given in local coordinates as follows: suppose $\omega \in \Omega^1(N)$. By use of a partition of unity, we may suppose that $\omega$ is supported in a coordinate domain. Then if $(x, \xi)$ and $(y, \eta)$ are coordinates on $U \subseteq M$, $V \subseteq N$ respectively, $\Phi$ has a coordinate representation $\phi^*(y, \eta) = (t_1(x, \xi), \dotsc, t_r(x, \xi), \theta_1(x, \xi), \dotsc, \theta_s(x, \xi))$, and $\omega = \sum_i g_i \, dy^i + \sum_j h_j \, d\eta^j$ for some functions $g_i(y, \eta), h_j(y, \eta)$. Then the pullback is defined by

$$
\Phi^*(\omega)  = \sum_i \phi^*(g_i) \, dt^i + \sum_j \phi^*(h_j) \, d\theta^j.
$$

Here $dt^i, d\theta^j$ are the differentials of $t^i, \theta^j$ respectively (see Defn. \ref{defdifferential}). As in the ungraded case, the Chain Rule implies that this definition is independent of the choice of coordinate system, hence the pullback is globally defined. The pullback map $\Phi^*: \Om^1(N) \to \Om^1(M)$ extends uniquely to a pullback map $\Phi^*: \Om^\bullet(N) \to \Om^\bullet(M)$ which preserves the $\N$-grading and the wedge product. Indeed, by again using partitions of unity, we can reduce the problem to the setting of a coordinate domain and then extend via the universal property of the exterior algebra. 

Locally, the pullback may be described on $k$-forms as follows: given coordinates $(x, \xi)$ on $U$ and $(y, \eta)$ on $V$, the morphism $\Phi$ has a coordinate description $\phi^*(y, \eta) = (t_1(x, \xi), \dotsc, t_r(x, \xi), \theta_1(x, \xi), \dotsc, \theta_s(x, \xi))$. Then the pullback $\Phi^*$ is given on a form $dy^{I} \wedge d\eta^{J} \cdot f$ by:
\be\label{localformpullback}
\Phi^*(dy^{I} \wedge d\eta^{J} \cdot f) = dt^I \wedge d\theta^{J} \cdot \phi^*(f)\;,
\ee
where multi-index notation is used, e.g., $dy^I:=dy^{i_1}\wedge \cdots \wedge dy^{i_p}$. 

Note that since $\Phi^*$ is a morphism of $\cO_N$-modules, it is in particular $\cJ$-adically continuous.

\bigskip
%------------------------------------------------
\subsection{Exterior derivative.}
%------------------------------------------------

Let $U$ be an open subsupermanifold of $M$.

\begin{defn}\label{defdifferential}
The {\it differential} of a function $f \in \cO(U)$ is the section $df$ of $\cT ^*M(U)$ defined by
$$
( X, df ) = Xf\;.
$$
\end{defn}

From the derivation property of $X$, we immediately obtain the Leibniz rule
$$
d(fg) = df \cdot g + f \cdot dg\;.
$$

It is immediately seen from its definition that the differential is compatible with the restriction maps of the sheaf $\cO_M$, since $\cT M$ and $\cO_M$ are both sheaves. Furthermore, the assignment $f \mapsto df$ is readily seen to be $\cJ$-adically continuous by $\cJ$-adic continuity of vector fields. Thus, the operator $d:\cO_M \to \cT ^*M$ is $\cJ$-adically continuous.

If $U$ is a coordinate chart with local coordinates $(x, \xi)$, the differential of $f$ is given in $U$ by the formula:
\begin{align}\label{formuladiff}
df = \sum_i dx^i \cdot \frac {\partial f} {\partial x^i} + \sum_j d\xi^j \cdot \frac {\partial f} {\partial \xi^j}\;.
\end{align}

As a consequence of the chain rule, formula \eqref{formuladiff} transforms correctly under a change of coordinates and is thus independent of the choice of coordinate system. Hence, it may be used as an alternative definition of the differential.

\begin{prop}
The differential $d: \mathcal{O} \to \mathcal{T}^*M$ is the universal derivation of $\Zn$-degree $0$ with values in an $\mathcal{O}$-module. 
\end{prop}

\begin{proof}
This is proven by composing several natural sheaf morphisms. First, for sheaves $\cG$ and $\cG'$, we have a natural morphism $\cG' \otimes_\cO \cG^* \to \underline{\mathcal{H}om}_\cO(\cG, \cG')$, which is given as follows: a basic element $s \otimes \varphi \in \cG'(U) \otimes_{\cO(U)} \cG^*(U)$ defines a homomorphism $\cG|_U \to \cG'|_U$ by $m \mapsto s \cdot \varphi(m)$. If $\cG$ is locally free of finite rank, this morphism is an isomorphism. Hence for a fixed locally free sheaf $\cG$ of finite rank, we have an isomorphism of sheaves $\cG' \otimes_\cO \cG^* \cong \underline{\mathcal{H}om}_\cO(\cG, \cG')$ which is functorial in $\cG'$.

We have another natural sheaf morphism $\cG \to (\cG^*)^*$, induced as follows: for $m \in \cG(U)$, we send $m \mapsto \varphi_m$, where $\varphi_m(f)= (-1)^{\langle deg(m), deg(f) \rangle} f(m)$ and if $\cG$ is locally free of finite rank, this is an isomorphism.

Finally, we specialize to the case $R = \cO(U)$, where $U$ is an open subset. There is a natural $\cO(U)$-module morphism $\phi_U:\cF(U)  \otimes_\cO \cT M(U)  \to Der_\R(\cO, \cF(U))$, given by $m \otimes X \mapsto [m] X$, where $([m]X)(f) := m \cdot Xf$ for $f \in \cO(U)$.

Now assume further that $U$ is a coordinate domain with coordinates $(u^i)$. In this case, we define a morphism $\psi_U: Der_\R(\cO, \cF(U)) \to  \cF(U) \otimes_\cO \cT M(U) $ by $D \mapsto \sum_i Du^i \otimes \partial_{u^i}$. We claim $\psi_U= \phi^{-1}_U$. The only significant point here is to show $\phi_U \circ \psi_U$ is the identity. This follows from a slight variant of the polynomial approximation arguments used in the proof of Prop. \ref{chainrule}.

The morphisms $\phi_U$ induce a natural sheaf morphism $\widetilde{\phi}:  \cF \otimes_\cO \cT M \to Der_\R(\cO, \cF)$. Since $\phi_U$ is an isomorphism if $U$ is a coordinate domain, $\widetilde{\phi}$ is a sheaf isomorphism which is natural in $\cF$.

Now let $\mathcal{F}$ be an $\mathcal{O}$-module. Noting that $\cT^*M$ is locally free of finite rank and putting together all the sheaf isomorphisms explained above, we have a sequence of isomorphisms

\begin{align*}
\underline{\mathcal{H}om}_\cO(\cT^*M, \mathcal{F}) & \cong  \mathcal{F} \otimes_\cO  (\cT^*M)^*  \\
&\cong  \mathcal{F} \otimes_\cO \cT M\\
&\cong Der_\R(\cO, \mathcal{F}),
\end{align*}

\noindent which are all functorial in $\cF$, so $\cT^*M$ satisfies the universal property characterizing the universal derivation up to unique isomorphism.
\end{proof}

As on ungraded manifolds, the differential on $\cO_M$ extends uniquely to a differential operator on the sheaf of differential forms $\Om^{\bullet}_M$ satisfying certain properties. First, we prove this globally.

\begin{prop}\label{uniqueextder}
For each $k \geq 0$, there exists a unique morphism of $\Zn$-super vector spaces $d: \Om^k(M) \to \Om^{k+1}(M)$ satisfying the following conditions:
\begin{enumerate}
\item For $k = 0$, $df$ is the differential of $f$.\
\item %($\Z_2^n\times\Z_2$-Graded derivation of degree $(\mathbf{0},1)$) 
 (Graded derivation) For any $k$-form $\alpha$ and any form $\beta$,
\[
d(\alpha \wedge \beta) = d\alpha \wedge \beta + (-1)^k \alpha \wedge d \beta.
\]
\item $d^2 = 0\,$.\\
\end{enumerate}
We call $d$ the {\bf exterior derivative}.
\end{prop}

\begin{proof}
We first show that there is at most one operator $d$ that satisfies the conditions of the proposition. Suppose such a $d$ exists. We first show $d$ is local, i.e. if $\omega$ is a $k$-form on $M$ that vanishes on an open subset $U$, $(d \omega)|_U = 0$. Let $u \in U$ be any point. By Proposition \ref{localization}, there exists a function $\varphi \in \cO^{0}(M)$ such that $supp(\varphi) \subset U$ and $\varphi \equiv 1$ in some neighborhood $V \subset U$ of $u$. Then $\varphi \omega \equiv 0$ on $M$, and we have $d(\varphi \omega) = d \varphi \wedge \omega + \varphi \cdot d\omega \equiv 0$. It follows from the properties of $\varphi$ that $(d \omega)|_V \equiv 0$. We conclude that if $\omega$ and $\omega'$ are $k$-forms on $M$ such that $\omega|_U = \omega'_U$, $d\omega|_U = d\omega'|_U$.

Let $x^i, \xi^j$ be coordinates on an open subsupermanifold $U$ of $M$, and let $\omega := \sum_{I,J} dx^I \wedge d\xi^J \cdot  \omega_{IJ}$ be a $|I| + |J|$-form on $U$. Let $u$ be any point in $U$. By Proposition \ref{localization}, there exist functions $\widetilde{x}^i$ (resp. $\widetilde{\xi}^j$) on $M$ which agree with $x^i$ (resp. $\xi^j$) in a neighborhood of $u$. In the same fashion, we may extend $\omega$ to a form $\widetilde{\omega} = \sum_{I,J} d\widetilde{x}^I \wedge d\widetilde{\xi}^J \cdot  \widetilde{\omega}_{IJ}$ on $M$, which agrees with $\omega$ near $u$.  By the graded derivation property (2) of $d$, we see that

\[
d \widetilde{\omega} = \sum_{I, J} d(d\widetilde{x}^I \wedge d \widetilde{\xi}^J) \widetilde{\omega}_{IJ} + (-1)^{|I| + |J|} \, d\widetilde{x}^I \wedge d \widetilde{\xi}^J \wedge d\widetilde{\omega}_{IJ}.
\]\

Note that $d(d\widetilde{x}^I \wedge d\widetilde{\xi}^J) = 0$ for any multiindices $I, J$; this follows from the fact that $d^2 x^i = d^2 \xi^j = 0$ by property (3), using induction on $|I| + |J|$ and the graded derivation property (2). Hence

\begin{equation} \label{defnextder}
d \widetilde{\omega} = \sum_{I, J} (-1)^{|I| + |J|} \, d\widetilde{x}^I \wedge d \widetilde{\xi}^J \wedge d\widetilde{\omega}_{IJ}.
\end{equation}\\

As $u$ was arbitrary, the asserted uniqueness follows, as $d$ is then given in $U$ by the right-hand side of Equation \eqref{defnextder}.

To prove existence, we first show that an operator $d$ satisfying properties (1)--(3) exists in the case where $M$ is a coordinate domain $U$ with coordinates $(x, \xi)$. We therefore {\it define} $d$ using Equation \eqref{defnextder}, for a form $\omega = \sum_{IJ} dx^I \wedge d\xi^J \cdot \omega_{IJ}$, where the indices in $I$ and $J$ are in standard order:

\[
d \widetilde{\omega} := \sum_{I, J} (-1)^{|I| + |J|} \, d\widetilde{x}^I \wedge d \widetilde{\xi}^J \wedge d\widetilde{\omega}_{IJ}.
\]\\

One may check that this formula for $d$ continues to hold with the indices not necessarily in standard order. We now show that $d$ so defined satisfies the desired properties. First, it is obvious from the definition that for any function $f$, $df$ is the differential of $f$.

It suffices to check the graded derivation property for forms $\omega = dx^I \wedge d\xi^J \cdot f, \eta = dx^K \wedge d\xi^L \cdot g$, where $|I| + |J|$. Let $\sigma := \#\{k \in K : \langle deg(x^k), deg(f) \rangle = 1\} +  \# \{ l \in L : \langle deg(\xi^l), deg(f) \rangle = 1\}$. We have:

\begin{align*}
&d(\omega \wedge \eta)\\
=& (-1)^{|I| + |J| + |K| + |L|} dx^I \wedge d\xi^J \wedge dx^K \wedge d\xi^L \wedge (-1)^{\sigma} (df \cdot g + f \cdot dg)\\[2.8mm]
&d\omega \wedge \eta + (-1)^{|I|+|J|} \omega \wedge d\eta \\
=& (-1)^{|I| + |J|} dx^I \wedge d\xi^J \wedge df \wedge dx^K \wedge d\xi^L \cdot g + (-1)^{|I|+|J|} dx^I \wedge d\xi^J \cdot f \wedge (-1)^{|K| + |L|} dx^K \wedge d\xi^L \cdot dg\\
=&(-1)^{|I| + |J| + |K| + |L|} dx^I \wedge d\xi^J \wedge dx^K \wedge d\xi^L \wedge (-1)^{\sigma} (df \cdot g + f \cdot dg).
\end{align*}

To show $d^2 = 0$, we begin by showing that $d^2f = 0$ for a function $f$ on $U$:

\begin{align*}
d^2f = & \sum_{i,k} dx^i \wedge dx^k \frac {\partial^2 f} {\partial x^k \partial x^i} + \sum_{i, l} dx^i \wedge d\xi^l \frac {\partial^2 f} {\partial \xi^l \partial x^i}\\
& + \sum_{j,m} d\xi^j \wedge dx^m \frac {\partial^2 f} {\partial x^m \partial \xi^j} + \sum_{j,n} d\xi^j \wedge d\xi^n \frac {\partial^2 f} {\partial \xi^n \partial \xi^j};
\end{align*}

\noindent it is a tedious but straightforward computation to check using $\Zn$-antisymmetry that the second and third sums cancel each other out, and that the first and fourth sums are both zero. (To show the fourth sum is zero, one also needs the fact that if $\xi^j$ is odd, $\partial^2 f / \partial (\xi^j)^2 = 0$).

For the general case, note first that $d(dx^I \wedge d\xi^J) = 0$ for any multiindices $I, J$; this follows from the fact that $d^2f = 0$ by induction on $|I| + |J|$ and the graded derivation property. Then, for a form of the type $dx^I \wedge d\xi^J \cdot  \widetilde{\omega}_{IJ}$, we have by the graded derivation property that:

\begin{align*}
d^2(dx^I \wedge d\xi^J \cdot  \widetilde{\omega}_{IJ}) &=d(dx^I \wedge d\xi^J) \wedge d\omega_{IJ} + (-1)^{|I| + |J|} dx^I \wedge d\xi^J \wedge d^2 \omega_{IJ}\\
&= 0.
\end{align*}

For the case of arbitrary $M$, note that $M$ is covered by coordinate charts $U$, on which $d$ is defined as above. If $\omega$ is a form on $U \cap U'$, $d_U\omega$ and $d_{U'}\omega$ agree on $U \cap U'$ by uniqueness, hence $d$ is globally well-defined on $M$.
\end{proof}

\begin{prop}
$d: \Omega^k(M) \to \Omega^{k+1}(M)$ is $\cJ$-adically continuous.
\end{prop}

\begin{proof}
Since $d$ is local, it suffices to verify the proposition locally in a coordinate chart $U$. The proof of Proposition \ref{dercont} implies that $d(\mathcal{J}^N(U)) \subseteq (\mathcal{J}^{N-1} \cdot \Omega^1)(U)$. Given a $k$-form $\omega = \sum_{IJ} dx^I \wedge d\xi^J \cdot  \omega_{IJ}$ such that $\omega_{IJ} \in \mathcal{J}^N(U)$, Formula \eqref{defnextder} then implies that $d\omega$ lies in $(\mathcal{J}^{N-1} \cdot \Omega^{k+1})(U)$. Exactly as in Prop. \ref{dercont}, one may then use this to prove $\cJ$-adic continuity of $d$ by checking that $d^{-1}(\eta + \cJ^N \cdot \Omega^{k+1}(U))$ is $\cJ$-adically open for any $\eta \in \Omega^{k+1}(U)$.
\end{proof}

The exterior derivative is compatible with pullbacks:

\begin{prop}\label{functextder}
Let $\Phi: M \to N$ be a morphism of $\Zn$-supermanifolds. Then for any $k$, $d(\Phi^*(\omega)) = \Phi^*(d \omega)$ for any $\omega \in \Omega^k(N)$, where $\Phi^*: \Om^\bullet(N) \to \Om^\bullet(M)$ is the pullback.
\end{prop}

\begin{proof}
Since $d$ is a local operator, we may assume $M$, $N$ are coordinate charts with coordinates $u=(x, \xi)$ and $v=(y, \eta)$ respectively. Then $\Phi$ is given by $\phi^*(x, \xi) = (t_1(x, \xi), \dotsc,t_r(x, \xi), \theta_1(x, \xi), \dotsc, \theta_s(x, \xi))$. 
We prove the proposition for $k = 0$ first. Using the local form of $\Phi^*$ \eqref{localformpullback} and the chain rule, for any function $f$, we have

\begin{align*}
\Phi^*(df) 
&= \Phi^*\left(\sum dv^i \partial_{v^i}f\right)\\
&= \sum_k dt_k \cdot \phi^*\left( \frac {\partial f} {\partial y^k} \right) + \sum_l d\theta_l \cdot \phi^* \left( \frac {\partial f} {\partial \eta^l} \right)\\
&= \sum_k \sum_i du^i \cdot \frac {\partial t_k } {\partial u^i}\, \phi^*\left( \frac {\partial f} {\partial y^k} \right) + \sum_l \sum_i du^i \cdot \frac {\partial \theta_l } {\partial u^i}\, \phi^* \left( \frac {\partial f} {\partial \eta^l} \right)\\
&= \sum_i du^i \cdot \frac {\partial \phi^*(f)} {\partial u^i} %+ \sum_j d\xi^j \cdot \frac {\partial \phi^*(f)} {\partial \xi^j} 
= d \Phi^*(f).
\end{align*}\\

For general $k$, any $k$-form is a linear combination of forms of the type $\omega = dy^{I} \wedge %\dotsc \wedge dy^{i_l} \wedge 
d\eta^{J} %\wedge \dotsc \wedge d\eta^{j_m} 
\cdot f$. Applying the $k = 0$ case, we have:

\begin{align*}
\phi^*(d \omega) &= \phi^*((-1)^{|I| + |J|} dy^{i_1} \wedge \dotsc \wedge dy^{i_l} \wedge d\eta^{j_1} \wedge \dotsc \wedge d\eta^{j_m} \wedge df)\\
&= (-1)^{|I| + |J|} dt_{i_1} \wedge \dotsc \wedge dt_{i_l} \wedge d\theta_{j_1} \wedge \dotsc \wedge d\theta_{j_m} \wedge \phi^*(df)\\
&=  (-1)^{|I| + |J|} d\phi^*(y^{i_1}) \wedge \dotsc \wedge d\phi^*(y^{i_l}) \wedge d\phi^*(\eta^{j_1}) \wedge \dotsc \wedge d\phi^*(\eta^{j_m}) \wedge d\phi^*(f)\\
&= d\phi^*(\omega).
\end{align*}
\end{proof}

By Proposition \ref{functextder}, $d$ commutes with the restriction homomorphisms $r^*_{UV}: \Omega^k(V) \to \Omega^k(U)$ where $U \subset V$ are open subsupermanifolds of a supermanifold $M$. Hence, $d$ turns $\Om^\bullet_M$ into a sheaf of differential $\Zn$-graded algebras such that the differential is $\cJ$-adically continuous.

\medskip
%------------------------------------------------
\subsection{Cotangent space.}
%------------------------------------------------

\begin{defn}
The {\it cotangent space} to $M$ at $m$ is the $\Zn$-graded $\R$-vector space $T^*_mM := \iHom_{\R}(T_mM, \R)$.
\end{defn}

If $\dim(M) = p|{\bf q}$, then $\dim_{\R}(T^*_mM) = p|{\bf q}$; indeed, given a coordinate system $u=(x, \xi)$ centered at a point $m$, the cotangent vectors $du|_m=(dx^i|_m, d\xi^a|_m)$ form a basis of $T_m^*M$.

\medskip
For the cotangent space, we have the analogues of previously proven propositions about the tangent space:
\begin{prop}
Let $\omega$ be a differential $1$-form defined in a neighborhood of $m$. Then $\omega$ induces a cotangent vector $\omega_m$ to $M$ at $m$. If $\omega$ is homogeneous, the degree of $\omega_m$ is the same as that of $\omega$.
\end{prop}

\begin{prop}
Let $m \in M$ be a point. Then the cotangent space $T^*_m M$ is isomorphic to the $\Zn$-graded $\R$-vector space $(\cT^*M) _m/\left(\mathfrak{m} \cdot (\cT ^*M)_m\right)$.
\end{prop}

The proofs are similar to those in the case of the tangent space, and are left to the reader.

\medskip
%------------------------------------------------
\subsection{Poincar\'e's lemma}
%------------------------------------------------

The {\it de Rham complex} on a $\Zn$-manifold $M$ is the complex $(\Om^\bullet_M, d)$, where $\Om^\bullet$ is the sheaf of differential forms and $d$ is the exterior differential.  We shall compute the cohomology of this complex.

\medskip

\begin{thm}[Poincar\'e's lemma for $\Zn$-supermanifolds]
Let $M$ be a $\Zn$-supermanifold. Then the de Rham complex $(\Om^\bullet_M, d)$ of $M$ is a resolution of the constant sheaf $\underline{\R}$.
\end{thm}

\medskip

Let us call the complex of $\Zn$-graded vector spaces $(\Om^\bullet(M), d)$ the {\it global de Rham complex}. We will begin with the following lemma (which is stated without proof in the $\Zs$-graded case in \cite{DM}).

\begin{lem}\label{dRtenprod}
The global de Rham complex of $\R^{p|{\bf q}}$ is isomorphic to the total tensor product complex (over $\R$) of the pullbacks of the global de Rham complexes of $\R^{p|{\bf q_{\bar{0}}}}$ and $\R^{0|{\bf q_{\bar{1}}}}$ via the projection maps to $\R^{p|{\bf q_{\bar{0}}}}$ and $\R^{0|{\bf q_{\bar{1}}}}$.
\end{lem}
In  this lemma, ${\bf q_{\bar{0}}}$ (resp. ${\bf q_{\bar{1}}}$) are the purely even (resp. the purely odd) part of ${\bf q}$, namely
\beas
{q_{\bar{0}, \gamma}}= \begin{cases} q_{\gamma} & \mbox{ if } \bar{\gamma}=\bar{0} \\ 0 & \mbox{ if } \bar{\gamma}=\bar{1}\end{cases}
&&
{q_{\bar{1},\mu}}= \begin{cases} 0 & \mbox{ if } \bar{\mu}=\bar{0} \\ q_{\mu} & \mbox{ if } \bar{\mu}=\bar{1}.\end{cases}
\eeas

\begin{proof}
 Let $\pi_+: \R^{p|{\bf q}} \to \R^{p|{\bf q_0}}$ and $\pi_-: \R^{p|{\bf q}} \to \R^{0|{\bf q_1}}$ denote the projection maps. Then $\pi^*_+(\omega) \otimes \pi^*_-(\eta) \mapsto \pi^*_+(\omega) \wedge \pi^*_-(\eta)$ defines a homomorphism of $\N$-graded $\Zn$-graded vector spaces $i: \pi^*_+(\Om^\bullet (\R^{p|{\bf q_0}})) \otimes_\R \pi^*_-(\Om^\bullet(\R^{0|{\bf q_1}})) \to \Om^\bullet(\R^{p|{\bf q}})$.

To show that $i$ is an isomorphism of $\mathbb{N}$-graded $\Zn$-super vector spaces, we find its inverse. Let $x^i, \psi^j$ be a coordinate system for $\R^{p|{\bf q_0}}$, and $\theta^k$ a coordinate system for $\R^{0|{\bf q_1}}$. Here the $x^i$ are of degree zero, $\psi^j$ are even of nonzero degree, and the $\theta^k$ are odd of nonzero degree. The pullbacks $\widetilde{x}^i := \pi^*_+(x^i), \widetilde{\psi}^j:= \pi^*_+(\psi^j), \widetilde{\theta}^k := \pi^*_-(\theta^k)$ constitute a coordinate system for $\R^{p|{\bf q}}$, hence the pullbacks $d\widetilde{x}^i = \pi^*_+(dx^i)$, $d\widetilde{\psi}^j = \pi^*_+(d\psi^j)$, $d\widetilde{\theta}^k = \pi^*_-(d\theta^k)$ constitute a global basis of the cotangent sheaf on $\R^{p|{\bf q}}$.

This implies that $\Omega^1(\R^{p|{\bf q}}) = \cO(\R^{p|{\bf q}}) \otimes_\R V$ as $\Zn$-super vector spaces, where $V$ is the $\Zn$-super vector space with basis $\{d \widetilde{x}^i, d \widetilde{\psi}^j, d \widetilde{\theta}^k\}$. Hence $\Om^\bullet(\R^{p|{\bf q}}) \cong \cO(\R^{p|{\bf q}}) \otimes_\R \Lambda^\bullet_\R(V)$. Similarly, $\pi^*_+(\Om^\bullet(\R^{p|{\bf q_0}})) \cong \pi^*_+(\cO(\R^{p|{\bf q_0}})) \otimes_\R \Lambda^\bullet_\R(V_+)$ and $\pi^*_-(\Om^\bullet(\R^{0|{\bf q_1}})) \cong \pi^*_-(\cO(\R^{0|{\bf q_1}})) \otimes_\R \Lambda^\bullet_\R(V_-)$.

We have $V = V_+ \oplus V_-$, where $V_+$ is the $\R$-span of $\{d \widetilde{x}^i, d \widetilde{\psi}^j \}$ and $V_-$ is the $\R$-span of $\{d\widetilde{\theta}^k\}$. Then the map $(v_+, v_-) \mapsto v_+ \otimes 1 + 1 \otimes v_-$ induces a natural isomorphism $\Lambda^\bullet_\R(V) \cong \Lambda^\bullet_\R(V_+) \otimes_\R \Lambda^\bullet_\R(V_-)$ by Corollary \ref{exteriordirsum}. 

Crucially, since $\R[\widetilde{\theta}^1, \dotsc, \widetilde{\theta}^{q_1}]$ is a polynomial algebra in the odd variables $\widetilde{\theta}^i$, we have that $\cO(\R^{p|{\bf q}}) \cong C^\infty(\R^p)[[\widetilde{\psi}^1, \dotsc, \widetilde{\psi}^{q_0}]] \otimes_\R \R[\widetilde{\theta}^1, \dotsc, \widetilde{\theta}^{q_1}] = \pi^*_+(\cO(\R^{p|{\bf q_0}})) \otimes_\R \pi^*_-(\cO(\R^{0|{\bf q_1}}))$, whence we have a sequence of isomorphisms

\begin{align*}
 \Om^\bullet(\R^{p|{\bf q}}) \cong & \cO(\R^{p|{\bf q}}) \otimes_\R \Lambda^\bullet_\R(V)\\
\cong &\cO(\R^{p|{\bf q}}) \otimes_\R \left[\Lambda^\bullet_\R(V_+) \otimes_\R \Lambda^\bullet_\R(V_-)\right]  \\
\cong &[ \pi^*_+(\cO(\R^{p|{\bf q_0}})) \otimes_\R \pi^*_-(\cO(\R^{0|{\bf q_1}}))] \otimes_\R \left[\Lambda^\bullet_\R(V_+) \otimes_\R \Lambda^\bullet_\R(V_-)\right]\\
\cong &\left[\pi^*_+(\cO(\R^{p|{\bf q_0}})) \otimes_\R \Lambda^\bullet_\R(V_+) \right] \otimes_\R \left[\pi^*_-(\cO(\R^{0|{\bf q_1}})) \otimes_\R \Lambda^\bullet_\R(V_-)\right] \\
\cong &\pi^*_+(\Om^\bullet (\R^{p|{\bf q_0}})) \otimes_\R \pi^*_-(\Om^\bullet(\R^{0|{\bf q_1}}))\\
\end{align*}

It is routine to check that the morphism $i$ defined above is left inverse to the composition of these isomorphisms, hence is also an isomorphism.

It is readily checked using coordinates that $\pi^*_+: \Om^\bullet(\R^{p|{\bf q_0}}) \to \Om^\bullet(\R^{p|{\bf q}})$ and $\pi^*_-: \Om^\bullet(\R^{0|{\bf q_1}}) \to \Om^\bullet(\R^{p|{\bf q}})$ are both injective. Letting $d_+$ (resp. $d_-$) denote the exterior differential on $\R^{p|{\bf q_0}}$ (resp. $\R^{0|{\bf q_1}}$), it follows we have well-defined differentials $\pi_+^* d_+$ on $\pi^*_+(\Om^\bullet(\R^{p|{\bf q_0}}))$ (resp. $\pi_-^* d_-$ on $\pi^*_-(\Om^\bullet(\R^{0|{\bf q_1}}))$, defined by

\[
(\pi^*_\pm d_{\pm})\omega_\pm:= \pi^*_{\pm}( d_{\pm} \sigma_\pm),
\]

\noindent for $\omega_+ \in \Om^\bullet(\R^{p|{\bf q_0}})$ (resp. $\omega_- \in \Om^\bullet(\R^{0|{\bf q_1}})$) where $\sigma_+$ (resp. $\sigma_-$) is the unique form in $\Om^\bullet(\R^{p|{\bf q_0}})$ (resp. $\Om^\bullet(\R^{0|{\bf q_1}})$) such that $\pi^*_\pm(\sigma_\pm) = \omega_\pm$.

The differentials $\pi^*_\pm d_\pm$ induce a differential $D$ on the $\Zn$-super vector space $\pi^*_+(\Om^\bullet (\R^{p|{\bf q_0}})) \otimes_\R \pi^*_-(\Om^\bullet(\R^{0|{\bf q_1}}))$ in a standard fashion:

\[
D(\omega \otimes \eta) := (\pi^*_+d_+) \omega \otimes \eta + (-1)^k \omega \otimes (\pi^*_-d_-) \eta,
\]\

\noindent where $\omega$ is a $k$-form and $\eta$ is any form, making $\pi^*_+(\Om^\bullet (\R^{p|{\bf q_0}})) \otimes_\R \pi^*_-(\Om^\bullet(\R^{0|{\bf q_1}}))$ into a complex of $\Zn$-super vector spaces.

Let us define a new differential $d'$ on $\Om^\bullet(\R^{p|{\bf q}})$ by $d' = i \circ D \circ i^{-1}$. We shall identify $d'$ with $d$. This could be done by direct computation, but we will show instead that $d'$ satisfies the axiomatic characterization of $d$ given in Proposition \ref{uniqueextder}.

First, $d'^2 =0$ since $D^2 = 0$. Now we will show that $d'f = df$ for all $f \in \cO({\R^{p|{\bf q}}})$. Note that any function in $\cO(\R^{p|{\bf q}})$ may be uniquely written as a finite sum of elements of the form $i(f_+ \otimes f_-)$ for unique $f_+ \in \pi^*_+(\cO(\R^{p|{\bf q_0}}))$ and $f_- \in \pi^*_-(\cO(\R^{0|{\bf q_1}}))$; by definition, $i(f_+ \otimes f_-) = f_+ \cdot f_-$. Let $g_+$ (resp. $g_-$) be the unique function in $\cO({\R^{p|{\bf q_0}}})$ (resp. $\cO({\R^{0|{\bf q_1}}})$) such that $\pi^*_\pm(g_\pm) = f_\pm$.

Then $d'[i (f_+ \otimes f_-)] = i[((\pi^*_+ d_+)f_+) \otimes f_- + f_+ \otimes (\pi^*_- d_-)f_-)] = \pi^*_+(d_+g_+) \cdot f_- + f_+ \cdot \pi^*_-(d_- g_-)$, which equals $d(f_+ \cdot f_-)$ by the Leibniz rule for $d$. Finally, the antiderivation property of $d'$ follows immediately from the definition of the induced differential $D$ on the tensor product. Hence $d'= d$, whence $i$ is a cochain map and thus an isomorphism of complexes.
\end{proof}

\bigskip

\begin{proof}[Proof of Poincar\'e's lemma]
Since the question is local, it suffices to prove the theorem for the global de Rham complex of $\R^{p|{\bf q}}$. In light of Lemma \ref{dRtenprod}, the K\"{u}nneth formula reduces the problem to proving the Poincar\'e lemma for $\R^{p|{\bf q_0}}$ and $\R^{0|{\bf q_1}}$ separately. Note that since $|\R^{p|{\bf q_0}}|$ and $|\R^{0|{\bf q_1}}|$ are both connected, the global sections of the sheaf $\underline{\R}$ in either case are just the constant functions and may be canonically identified with the real numbers $\R$.

\bigskip

\noindent \underline{$\R^{p|{\bf q_0}}$}: Here, we cannot directly reduce the global de Rham complex to the tensor product of the pullbacks of the global de Rham complexes of the factors as we did previously. Instead, we write $\R^{p|{\bf q_0}}$ as the product $\R^{p|{\bf q_0}'} \times \R^{0|{\bf i}}$, where ${\bf i}$ denotes the dimension $0|\dotsc |1| \dotsc |0$, and ${\bf q_0'}$ denotes the codimension of ${\bf i}$ in ${\bf q_0}$. Let $\pi: \R^{p|{\bf q_0}} \to \R^{0|{\bf q_0}'}$ be the projection morphism, and $s: \R^{p|{\bf q_0}'} \to \R^{0|{\bf q_0}}$ the zero-section morphism. Let $x, \psi$ denote (schematically) the coordinates on $\R^{p|{\bf q_0}'}$, and $\eta$ the coordinate on $\R^{0|{\bf i}}$. Then $\pi(x, \psi, \eta) = (x, \psi)$, and $s(x, \psi) = (x, \psi, 0)$. We will show that the maps of complexes $\pi^*$ and $s^*$ induce isomorphisms on cohomology by defining a homotopy operator $K: \Omega^k(\R^{p|{\bf q_0}}) \to \Omega^{k-1}(\R^{p|{\bf q_0}})$ as follows. Any differential $k$-form on $ \R^{p|{\bf q_0}}$ may be uniquely written as a sum of forms of two types:

\begin{align*}
\text{A) } & \pi^*(\sigma) \cdot f(x, \psi, \eta)\\
\text{B) } & \pi^*(\sigma) \wedge d\eta \cdot f(x, \psi, \eta),
\end{align*}\

\noindent where $\sigma$ is a form on $\R^{p|{\bf q_0'}}$ and $f(x, \psi, \eta)$ a function on $\R^{p|{\bf q_0}}$. Define:

\begin{align*}
& K(\pi^*(\sigma) \cdot f(x, \psi, \eta)) := 0\\
& K(\pi^*(\sigma) \wedge d\eta \, f(x, \psi, \eta)) := \pi^*(\sigma) \cdot F(x, \psi, \eta),
\end{align*}\

\noindent where $F(x, \psi, \eta)$ is the unique function in $C^\infty(\R^p)[[\psi, \eta]]$ such that

\begin{align*}
&\partial F/\partial \eta = f\\
&F(x, \psi, 0) = 0.
\end{align*}\

It is easy to see that such an $F$ is unique, if it exists. To show existence, note that $f$ may be written as $\sum_{k=0}^\infty \eta^k a_k(\psi)$, where the $a_k$ are uniquely determined formal power series in the coordinates $\psi$ with coefficients in $C^\infty(\R^p)$. Then $F := \sum_{k=0}^\infty \eta^{k+1} \frac {a_k(\psi)} {k+1}$ is an antiderivative of $f$ with respect to $\eta$, and $F(x, \psi, 0) = 0$. Hence the operator $K$ is well-defined.

\medskip

For $\omega$ of type A), we have:

\begin{align*}
(dK - Kd)(\omega) &= -K\bigg(\pi^*(d\sigma) \cdot f + (-1)^k\pi^*(\sigma) \bigg[ dx^i \cdot \frac {\partial f} {\partial x^i} + d\psi^j \cdot \frac {\partial f} {\partial \psi^j} + d\eta \cdot \frac {\partial f} {\partial \eta} \bigg] \bigg) \\
&=(-1)^{k-1} \pi^*(\sigma)[f - f(x, \psi, 0)],
\end{align*}\

\noindent and $(id - \pi^* \circ s^*)(\omega) = \pi^*(\sigma)[f - f(x, \psi, 0)]$.

\medskip

For $\omega$ of type B), we have:

\begin{align*}
&(dK - Kd)(\omega) \\
= &\pi^*(d\sigma) \cdot F + (-1)^{k-1}\pi^*(\sigma) \wedge \bigg[ dx^i \cdot \frac {\partial F} {\partial x^i} + d\psi^j \cdot \frac {\partial F} {\partial \psi^j}  + d\eta \cdot f \bigg] \\
-&\pi^*(d\sigma) \cdot F -(-1)^{k-1} \pi^*(\sigma) \wedge \bigg[dx^i \cdot \frac {\partial F} {\partial x^i} + d \psi^j \cdot \frac {\partial F} {\partial \psi^j} \bigg] \\
=& (-1)^{k-1}\omega,
\end{align*}\

\noindent and $(id - \pi^* \circ s^*)(\omega) = \omega$.

We have just shown that $\pi^* \circ s^*$ is homotopic to (a multiple of) the identity on $k$-forms; since $\pi \circ s = id$, $s^* \circ \pi^* = id$. It follows that $\pi^*$ and $s^*$ are chain homotopy inverses, hence the global de Rham complex $(\Om^\bullet(\R^{p|{\bf q_0}}), d)$ is chain-homotopy equivalent to $(\Om^\bullet(\R^{p|{\bf q_0}'}), d)$. By induction on $q_0$, the problem is thus reduced to the classical Poincar\'e lemma for $\R^{p}$.\\

\medskip

\noindent \underline{$\R^{0|{\bf q_1}}$}: Note that $\R[\theta^1, \dotsc, \theta^{q_1}] \cong \otimes_{i=1}^{q_1} \R[\theta^i]$ as $\Zn$-superalgebras over $\R$. Hence the global de Rham complex of $\R^{0|{\bf q}}$ is the tensor product over $\R$ of the pullbacks of the global de Rham complexes of the projections from $\R^{0|{\bf q_1}}$ onto the individual factors; the argument is entirely analogous to the one given in the proof of Lemma \ref{dRtenprod} and is left to the reader. We may thus reduce to the case of $\R^{0|{\bf j}}$ where ${\bf j} = 0|\dotsc |1| \dotsc |0$. Let $\theta$ be the coordinate on $\R^{0|{\bf j}}$, then the de Rham complex of $\R^{0|{\bf j}}$ is

\[
0 \to \R \xrightarrow{i} \R[\theta] \xrightarrow{d} \Omega^1 \xrightarrow{d} \Omega^2 \to \dotsc.
\]\

It is obvious that the functions $f \in \R[\theta]$ such that $df = 0$ are precisely the constants. For $k \geq 1$, any $k$-form $\omega$ may be written as $\omega = (d\theta)^k (a + \theta b)$, for $a, b \in \R$. $\omega$ is closed if and only if $b = 0$. Then $\sigma := (d \theta)^{k-1} \cdot \theta a$ is an antiderivative for $\omega$.

\end{proof}

\begin{rema}
By the existence of partitions of unity on $\Zn$-supermanifolds \cite{CGPa}, the sheaf of functions $\cO$ on a $\Zn$-supermanifold $M$ is soft, whence the same is true for the sheaves of $\cO$-modules $\Omega^k_M$. Hence for any $k$ the \v{C}ech cohomology groups $H^i(M, \Omega^k)$ vanish for all $i > 0$. By standard arguments, we have the de Rham isomorphism for any $\Zn$-supermanifold $M$:

\begin{align*}
H^*_{dR} (M, \Om^\bullet) \cong H^*(|M|, \R).
\end{align*}\

Thus the de Rham cohomology of a $\Zn$-supermanifold recovers the topological cohomology of the underlying reduced space, just as in ordinary supergeometry.
\end{rema}

%-%-%-%-%-%-%-%-%-%-%-%-%-%-%
\section{Appendix}
%-%-%-%-%-%-%-%-%-%-%-%-%-%-%

%------------------------------------------------
\subsection{Linear algebra over Hausdorff complete $\Zn$-supercommutative rings.}
%------------------------------------------------

In this section, we will work exclusively with $\Zn$-supercommutative rings which are Hausdorff complete in the $J$-adic topology, where $J$ denotes the homogeneous ideal of $R$ generated by the elements of nonzero degree.

\subsubsection{Rank of a linear map.}

We begin with the following criterion for invertibility of a square degree-zero matrix.\\

\begin{prop}\label{invertibleblockdiag}
Let $R$ be a $J$-adically Hausdorff complete $\Zn$-supercommutative ring. Let $T$ be a degree zero square $p|{\bf q}$ matrix with entries in $R$, written in the standard block format:

\begin{equation*}
T =  \left(\begin{array}{c|c|c}
T_{11} & \dotsc & T_{1N} \\
\hline
\vdots & \ddots & \vdots \\
\hline
T_{N1} & \dotsc & T_{NN} \\
\end{array}\right).
\end{equation*}

\noindent Then $T$ is invertible if and only if  $T_{ii}$ is invertible for all $i$.
\end{prop}

\begin{proof}

The algebra morphism $\varepsilon: A \to \,A/J$ %-- where $J$ is the ideal generated by the non-zero degree elements --%
induces a map on matrices which sends a zero-degree matrix $T$ over $A$ to the block-diagonal matrix $\overline{T}:=(t^i_j \mod J)$ over $A/J$.

Clearly, if $T$ is invertible with inverse $T^{-1}$, then $\overline{T}$ is also invertible with inverse $\overline{T}^{-1}=\overline{(T^{-1})}$.

Conversely, let us assume $\overline{T}$ is invertible. Then, there exists a zero-degree matrix $Y$ such that $TY=\mathbb{I}+Z$, with $Z\in \mathrm{gl}^0(\vec{r};J)$. By Hausdorff-completeness, $\mathbb{I}+Z$  is invertible with inverse
$$ (\mathbb{I}+Z)^{-1}= \mathbb{I}+\sum_{k\geq 1} (-Z)^k $$
Indeed, the right-hand side is meaningful: it is the matrix whose $(i,j)$-th entry (for every pair of indices $i,j$) is the unique limit of the ($J$-adic) Cauchy sequence of partial sums
$$ \Big(
%(\mathbb{I}-Z+Z^2 - \ldots +(-Z)^k)_{ij} %=
\delta_{ij}-z_{ij} + \sum_{l}z_{il}z_{lj}+ \ldots +(-1)^k \!\!\sum_{a_1,\ldots, a_{k-1}} \! z_{i,a_1}z_{a_1,a_2}\cdots z_{a_{k-1}j}
\Big)_k$$
Invertibility of $T$ follows.
\end{proof}

\bigskip

\begin{defn}
Let $R$ be a $\Zn$-superalgebra such that the group of units $R^*$ is a subset of $R^0$, and let $T \in Hom_R(R^{p|{\bf q}}, R^{m|{\bf n}})$. The {\it rank} of $T$, denoted $\text{rank} \, (T)$, is the graded dimension of the largest invertible submatrix of $T$.
\end{defn}

\medskip

\begin{rema}
We emphasize that the notion of $\mathbb{Z}^n_2$-graded rank is only well-behaved for matrices over a $\Zn$-supercommutative ring $R$ for which $R^* \subseteq R^0$; for arbitrary $\Zn$-supercommutative rings, one has only the classical rank and super rank of a matrix (cf. \cite{CM} where various notions of rank are discussed for $\Gamma$-graded algebras, $\Gamma$ a finite abelian group). The structure sheaf $\cO_M$ of a $\Zn$-supermanifold $M$ is clearly a sheaf of $\Zn$-supercommutative rings that satisfy the condition $R^* \subseteq R^0$.
\end{rema}

\begin{prop}
Let $R$ be a $\Zn$-superalgebra such that $R^* \subseteq R^0$, and suppose $R$ is $J$-adically Hausdorff complete. Let $T \in End_R(R^{p|{\bf q}})$ be in the standard block format. Let $T_{ii}$ be the $i$th block submatrix (possibly empty) of degree zero elements of $T$, $i = 1, \dotsc, N$. Then $\text{rank} \, (T) = \text{rank} \, (T_{11}) \, |\,\text{rank} \, (T_{22}) \, | \dotsc | \, \text{rank} \, (T_{NN})$, where the rank of the empty matrix is defined to be $0$ by convention.
\end{prop}

\begin{proof}
Let $\text{rank}(T) = r|{\bf s}$. Then there exists an invertible $r|{\bf s}$ square submatrix $T'$ of $T$. By Proposition \ref{invertibleblockdiag}, the block-diagonal degree-zero submatrices $T'_{ii}$ are all invertible. $T'_{11}$ is an $r \times r$ square matrix, and $T'_{ii}$ is an $s_{i-1} \times s_{i-1}$ square matrix, whence $r \leq \text{rank} \, (T_{11})$, $s_{i-1} \leq \text{rank} \, (T_{ii})$ for $i = 2, \dotsc, N$. On the other hand, for each $i$ we have that some $\text{rank} \,(T_{ii}) \times \text{rank} \, (T_{ii})$ submatrix of $T_{ii}$ is invertible. Deleting all rows and columns of $T$ which do not contain any entries from these invertible submatrices produces a $\text{rank} \, (T_{11}) \, |\,\text{rank} \, (T_{22}) \, | \dotsc | \, \text{rank} \, (T_{NN})$ square submatrix of $T$ which is invertible by Proposition \ref{invertibleblockdiag}. We conclude that $r | {\bf s} = \text{rank} \, (T_{11}) \, |\,\text{rank} \, (T_{22}) \, | \dotsc | \, \text{rank} \, (T_{NN})$.
\end{proof}

\bigskip

%------------------------------------------------
\subsubsection{Free modules}
%------------------------------------------------

We have the following simple but very important property of free modules over $J$-adically Hausdorff complete $\Zn$-graded rings.

\bigskip

\begin{prop} \label{freecomplete}
Let $T$ be a $J$-adically Hausdorff complete $\Zn$-supercommutative ring, and let $M$ be a free $R$-module. Then $M$ is $J$-adically Hausdorff complete.
\end{prop}

\begin{proof}
Since $M$ is free, there exists a homogeneous basis $e_i$ of $M$. We have $J^k M = J^k \bigoplus_i R e_i = \bigoplus_i J^k e_i$. This reduces the proposition to the case of a free module on a single homogeneous basis element; in this case, the proposition follows from the $J$-adic Hausdorff completeness of $R$ as a module over itself.
\end{proof}

\bigskip

\begin{comment}
\subsubsection{Tensor product}

Let $M, N$ be $R$-modules, let $M \otimes_R N$ denote their tensor product considered as an $R$-module, and denote the topological module $(M \otimes_R N, I)$ by $M \otimes_{R, I} N$. We endow $M \times N$ with the product topology coming from $(M, I)$ and $(N, I)$. We will show that $Hom_{R, I}(M, \iHom_R(N, P)) \cong Hom_{R, I}(M \otimes_{R, I} N, P)$, where $Hom_{R, I}$ denotes the set of $I$-adically continuous $R$-module homomorphisms, and $\iHom_R(- , -)$ denotes the internal $\iHom$ endowed with the $I$-adic topology. Thus, $M \otimes_{R, I} N$ is the tensor product in the category of $(R, I)$-modules.
\end{comment}

\subsection{The symmetric algebra and the exterior algebra}

We will require some basic algebraic results on the symmetric and exterior algebras. Our treatment follows the point of view of Deligne-Morgan \cite{DM}.

Let $R$ be a commutative $\Zn$-superalgebra. Then the category $Mod_R$ of $\Zn$-graded $R$-modules is a symmetric monoidal category with monoidal structure given by the tensor product $\otimes$ over $R$, with the braiding $v \otimes w \mapsto (-1)^{\langle \text{deg}(w), \text{deg}(v) \rangle} w \otimes v$. Hence for any $R$-module $M$, there is a canonical action of the symmetric group $S_n$ on the $n$-fold tensor product $M^{\otimes n}$.

The $n$th symmetric power of $M$, denoted $Sym^n_R(M)$, is the quotient of $M \otimes \dotsc \otimes M$ ($n$ times) by the action of the symmetric group $S_n$. Since infinite direct sums exist in the category of $R$-modules, the symmetric algebra $Sym^\bullet_R(M) := \oplus_{n \geq 0} Sym^n_R(M)$ is well-defined and is a commutative unital $\Zn$-superalgebra over $R$, the product being induced by the tensor product:

\begin{align*}
[m_1 \otimes \dotsc \otimes m_k] \cdot [n_1 \otimes \dotsc \otimes n_l] = [m_1 \otimes \dotsc \otimes m_k \otimes n_1 \otimes \dotsc \otimes n_l].
\end{align*}\\

There is a canonical $R$-module morphism $M \to Sym^\bullet_R(M)$, given by naturally identifying $M$ with $Sym^1_R(M)$. The symmetric algebra is characterized up to unique isomorphism by the following universal property: for any commutative unital $R$-algebra $A$ and $R$-module morphism $f: M \to A$, there exists a unique $R$-algebra morphism $F$ such that the diagram

\[
\begin{tikzcd}
M\arrow{r}{i} \arrow{dr}{f} & Sym^\bullet_R(M) \arrow{d}{F}\\
& A
\end{tikzcd}
\]

\bigskip

\noindent commutes.

Let $R$ be a $\Zn$-superalgebra over $\R$, and let $Mod^{gr}_R$ be the symmetric monoidal category of $\mathbb{Z}$-graded $R$-modules endowed with the tensor product and associativity isomorphisms of $Mod_R$, but with the braiding given by:

\begin{align*}
c_{VW}^{gr}: &V \otimes W \to W \otimes V\\
&v \otimes w \mapsto (-1)^{mn} c_{VW}(v \otimes w)
\end{align*}\

\noindent where $c_{VW}$ is the braiding in $Mod_R$ and $v \in V^m, w \in W^n$.

For $M \in Mod_R$, we define the exterior algebra $\Lambda^\bullet(M)$ to be the symmetric algebra (in $Mod^{gr}_R$) on the object of $Mod^{gr}_R$ whose $\mathbb{Z}$-degree $1$ part equals $M$ and which is $0$ in all other degrees. (This definition is readily seen to agree with that of \cite{CM}). The degree $k$ part of $\Lambda^\bullet(M)$ is just $\Lambda^k(M)$.

The universal property of the symmetric algebra then implies the crucial

\begin{cor}\label{freeexterior}
If $M$ is a free $R$-module of finite rank, $\Lambda^k(M)$ is a free $R$-module of finite rank for each $k$.
\end{cor}

We will also require:

\begin{cor}\label{exteriordirsum}
Let $M', M''$ be $R$-modules, and let $M:= M' \oplus M''$. Then there is a natural isomorphism $ F_M: \Lambda^\bullet(M) \to  \Lambda^\bullet(M') \otimes_R \Lambda^\bullet(M'')$, induced by $(m', m'') \mapsto m' \otimes 1 + 1 \otimes m''$, such that if $f': M' \to N', f'': M''\to N''$ are morphisms, and $N = N' \oplus N''$, the diagram

\[
\begin{tikzcd}
\Lambda^\bullet(M) \arrow{r}{F_M} \arrow{d}{(f' \oplus f'')_*} &\Lambda^\bullet(M') \otimes_R \Lambda^\bullet(M'') \arrow{d}{f'_* \otimes f''_*}\\ \Lambda^\bullet(N) \arrow{r}{F_N} &\Lambda^\bullet(N') \otimes_R \Lambda^\bullet(N'')
\end{tikzcd}
\]

is commutative.
\end{cor}

\begin{proof}
We note the following sequence of isomorphisms:

\begin{align*}
& Hom_{alg}^{gr}(Sym^{\bullet, gr}(M') \otimes_R^{gr} Sym^{\bullet, gr}(M''), A)\\
\cong & Hom_{alg}^{gr}(Sym^{\bullet, gr}(M'), A) \prod Hom_{alg}(Sym^{\bullet, gr}(M''), A)\\
\cong & Hom_R(M', A) \prod Hom_R(M'', A)\\
\cong & Hom_R(M' \oplus M'', A)\\
\cong & Hom_R^{gr} (Sym^{\bullet, gr}(M), A)\\
\end{align*}

All isomorphisms involved are canonical: the first results from the fact that the tensor product is the coproduct of $\Zn$-superalgebras over $R$ in $Mod^{gr}_R$; the second results from the universal property of the symmetric algebra in $Mod^{gr}_R$; the third is the compatibility of $Hom_R$ (considered as a contravariant functor in the first variable) with coproducts; the fourth is the universal property of $Sym^{\bullet, gr}_R(M)$.

Hence the composition of the above sequence of isomorphisms induces a natural equivalence between the functors of $A$-points of $Sym^{\bullet, gr}_R(M') \otimes^{gr}_R Sym^{\bullet, gr}_R(M'')$ and $Sym^{\bullet, gr}(M)$; by Yoneda's lemma, the inverse of this equivalence corresponds to an isomorphism $Sym^{\bullet, gr}_R(M) \to Sym^{\bullet, gr}_R(M') \otimes^{gr}_R Sym^{\bullet, gr}_R(M'')$. By substituting $Sym^{\bullet, gr}_R(M') \otimes^{gr}_R Sym^{\bullet, gr}_R(M'')$ for $A$ in the above, and considering the identity map of $Sym^{\bullet, gr}_R(M') \otimes^{gr}_R Sym^{\bullet, gr}_R(M'')$, the reader may check that this isomorphism is precisely $F_M$ and verify the stated commutativity of the diagram.
\end{proof}

\bigskip

%%%%%%%%%%%%%%%%%%%%%%%%%%%%%%%%%%%%%%%%%%%%%%%%%%%%%%%%%%%%%%%%%%%%%%%%%%%%%%%%%%%%%%%%%%%%%%%%%%%%%%%%%%%%%%%%%%%%%%%%%%
%%%%%%%%%%%%%%%%%%%%%%%%%%%%%%%%%%%%%%%%%%%%%%%%%%%%%%%%%%%%%%%%%%%%%%%%%%%%%%%%%%%%%%%%%%%%%%%%%%%%%%%%%%%%%%%%%%%%%%%%%%

\end{document}